\documentclass{lmcs}
\pdfoutput=1

\usepackage{lastpage}
\lmcsdoi{20}{4}{8}
\lmcsheading{}{\pageref{LastPage}}{}{}%
{Jul.~26,~2024}{Oct.~29,~2024}{}

\usepackage[utf8]{inputenc}

\usepackage{amsmath, amssymb, amsthm}
\usepackage{mathrsfs} 
\usepackage{multicol} 
\usepackage{hyperref}
\usepackage{float}
\usepackage{xcolor}
\usepackage{stmaryrd}
\usepackage{tikz}
\usepackage{graphicx}
\graphicspath{ {./figures/} }

\usepackage{tikzit}

\tikzstyle{edge}=[fill=white, draw=black, shape=circle]
\tikzstyle{node}=[fill=black, draw=black, shape=circle, inner sep=1.5pt]
\tikzstyle{morphism}=[fill=white, draw=black, shape=rectangle]

\tikzstyle{pointy}=[->]
\tikzstyle{dashy}=[-, dashed]
\tikzstyle{dashpoint}=[dashed, ->]
\tikzstyle{bluefill}=[-, fill={rgb,255: red,190; green,240; blue,255}]
\tikzstyle{greyfill}=[-, fill={rgb,255: red,240; green,240; blue,240}]

\newenvironment{theorem}{%
    \begin{thm}%
}{%
    \end{thm}%
    \ignorespacesafterend%
}
\newenvironment{proposition}{%
    \begin{prop}%
}{%
    \end{prop}%
    \ignorespacesafterend%
}
\newenvironment{definition}{%
    \begin{defi}%
}{%
    \end{defi}%
    \ignorespacesafterend%
}

\newenvironment{remark}{%
    \begin{rem}%
}{%
    \end{rem}%
    \ignorespacesafterend%
}
\newenvironment{example}{%
    \begin{exa}%
}{%
    \end{exa}%
    \ignorespacesafterend%
}

\newcommand{\topicsentence}[1]{#1}

\newcommand{\Nat}[0]{\ensuremath{\mathbb{N}}}
\newcommand{\Real}[0]{\ensuremath{\mathbb{R}}}
\newcommand{\Complex}[0]{\ensuremath{\mathbb{C}}}
\newcommand{\Bool}[0]{\ensuremath{\mathbb{B}}}

\newcommand{\defeq}[0]{\ensuremath{:=}}
\newcommand{\deftext}[1]{\textbf{#1}}
\newcommand{\functor}[1]{\ensuremath{\mathsf{#1}}}
\newcommand{\cp}[0]{\ensuremath{\fatsemi}} 
\newcommand{\id}[0]{\ensuremath{\mathsf{id}}}
\newcommand{\twist}[0]{\ensuremath{\sigma}}

\newcommand{\cat}[1]{\ensuremath{\mathscr{#1}}}

\newcommand{\Functor}[1]{\ensuremath{\mathsf{#1}}}

\newcommand{\unit}[0]{I}
\newcommand{\assoc}[0]{\ensuremath{\alpha}}
\newcommand{\assocInv}[0]{\ensuremath{\alpha}^{-1}}
\newcommand{\unitl}[0]{\ensuremath{\lambda}}
\newcommand{\unitlInv}[0]{\ensuremath{\lambda}^{-1}}
\newcommand{\unitr}[0]{\ensuremath{\rho}}
\newcommand{\unitrInv}[0]{\ensuremath{\rho}^{-1}}
\newcommand{\bigadapter}[0]{\ensuremath{\Phi}}
\newcommand{\bigadapterInv}[0]{\ensuremath{\Phi^{*}}}
\newcommand{\adapter}[0]{\ensuremath{\phi}}
\newcommand{\adapterInv}[0]{\ensuremath{\phi^{*}}}

\newcommand{\CatC}[0]{\ensuremath{\cat{C}}}
\newcommand{\TensorC}[0]{\ensuremath{\otimes}}
\newcommand{\unitC}[0]{\ensuremath{\unit_{\CatC}}}

\newcommand{\generator}[1]{\scalebox{0.4}{\input{#1.tikz}}}

\newcommand{\stikzfig}[1]{\scalebox{0.7}{\input{#1.tikz}}}

\newcommand{\CatW}[0]{\ensuremath{\cat{W}}}
\newcommand{\unitW}[0]{\ensuremath{\unit_{\CatW}}}
\newcommand{\TensorW}[0]{\ensuremath{\otimes}}

\newcommand{\CatSW}[0]{\ensuremath{\overline{\CatW}}}
\newcommand{\unitSW}[0]{\ensuremath{\unit_{\CatSW}}}

\newcommand{\CatD}[0]{\ensuremath{\overline{\CatC}}}
\newcommand{\TensorD}[0]{\ensuremath{\bullet}}
\newcommand{\unitD}[0]{\ensuremath{\unit_{\CatD}}}
\newcommand{\assocD}[0]{\ensuremath{\assoc_{\CatD}}}

\newcommand{\ObjD}[1]{{\ensuremath{\overline{#1}}}}
\newcommand{\MorD}[1]{{\ensuremath{\overline{#1}}}}

\newcommand{\FunF}[0]{\ensuremath{\mathcal{S}}}
\newcommand{\FunG}[0]{\ensuremath{\mathcal{N}}}
\newcommand{\FunU}[0]{\ensuremath{\Functor{U}}}

\newcommand{\FunS}[0]{\FunF} 
\newcommand{\FunN}[0]{\FunG} 

\newcommand{\objsize}[0]{\ensuremath{\mathtt{size}}}
\newcommand{\objpack}[0]{\ensuremath{\mathtt{pack}}}
\newcommand{\objunpack}[0]{\ensuremath{\mathtt{unpack}}}

\newcommand{\canonical}[0]{\ensuremath{\mathsf{can}}}
\newcommand{\dom}[0]{\ensuremath{\mathsf{dom}}}
\newcommand{\cod}[0]{\ensuremath{\mathsf{cod}}}

\newcommand{\CatT}[0]{\ensuremath{\mathbf{1}}}
\newcommand{\CatM}[0]{\ensuremath{\cat{M}}}
\newcommand{\TensorM}[0]{\ensuremath{\otimes}}
\newcommand{\unitM}[0]{\ensuremath{\unit_{\CatM}}}
\newcommand{\ItM}[0]{\ensuremath{\mathsf{It}(\cat{M})}}
\newcommand{\unitItM}[0]{\ensuremath{\Functor{Const}_{\unitM}}}
\newcommand{\TensorItM}[0]{\ensuremath{\Box}}

\theoremstyle{thmC}
\newtheorem{corollaryC}[thm]{Corollary}
\begin{document}

\title{String Diagrams for Strictification and Coherence}
\author{Paul Wilson\lmcsorcid{https://orcid.org/0000-0003-3575-135X}}[a]
\author{Dan {Ghica}}[b]
\author{Fabio {Zanasi}\lmcsorcid{https://orcid.org/0000-0001-6457-1345}}[c, d]

\address{Independent, United Kingdom}
\email{paul@statusfailed.com}

\address{University of Birmingham, United Kingdom}
\email{d.r.ghica@cs.bham.ac.uk}

\address{University College London, United Kingdom}
\email{f.zanasi@ucl.ac.uk}

\address{University of Bologna, OLAS team (Inria), Italy}

\begin{abstract}
  \label{section:abstract}
  Whereas string diagrams for strict monoidal categories are well understood, and have found application in several fields of Computer Science,
 graphical formalisms for non-strict monoidal categories are far less studied.
In this paper, we provide a presentation by generators and relations of string
diagrams for non-strict monoidal categories,
and show how this construction can handle applications in domains such as
digital circuits and programming languages.
We prove the correctness of our construction, which yields a novel proof of Mac
Lane's strictness theorem.
This in turn leads to an elementary graphical proof of Mac Lane's
\emph{coherence} theorem, and in particular allows for the inductive
construction of the canonical isomorphisms in a monoidal category.

\end{abstract}

\maketitle

\section{Introduction}
\label{section:introduction}
String diagrams are a rigorous graphical notation for morphisms in a category, which is proving useful in a broad variety of application domains, such as quantum systems~\cite{PQP}, computational linguistics~\cite{DBLP:journals/apal/CoeckeGS13}, digital circuits~\cite{DBLP:conf/csl/GhicaJL17}, or signal flow analysis~\cite{DBLP:series/ifip/Bonchi0Z21} --- see e.g.~\cite{PiedeleuZanasi23} for a recent overview.
What the majority of string diagrammatic notations have in common is that they are devised for monoidal categories in which the tensor is \emph{strict}, i.e. the associator and unitor morphisms are identities.
As Joyal and Street explain in their seminal \emph{Geometry of Tensor Calculus}~\cite{JOYAL199155}, the choice of using a strict monoidal category was motivated by convenience (``simplicity of exposition'') and by a wish to focus on ``aspects other than the associativity of tensor product''. 
Furthermore, they believed that ``most results obtained with the hypothesis that a tensor category is strict can be reformulated and proved without this condition.''

Indeed, in terms of mathematical power, this statement is true. 
However, string diagrams have been used increasingly as a convenient \emph{syntax} for languages with models in (strict) monoidal categories. 
And, when used as syntax, the distinction between strict and non-strict tensor becomes relevant, if not in terms of mathematical expressiveness then at least as a mechanism of abstraction. 
This is why modern programming languages, and even some modern hardware design languages such SystemVerilog~\cite{sutherland2006systemverilog}, use non-strict features such as \emph{tuples} and \emph{structs} which can nest in non-trivial ways.
These non-strict structures could be manually `strictified' by the programmer by flattening them into arrays.
Using such programmer conventions instead of native syntactic support does not entail a loss of expressiveness, but a loss of code readability, convenience, and general programmer effectiveness.

In this paper, we address the problem of expanding the graphical language of string diagrams with the required features that allow the expression of non-strict tensors. 
What makes the language of strict tensors convenient for the graphical representation is that objects are naturally represented as \emph{lists of wires}. 
This suggests that string diagrams make use of strictness in an essential way and, indeed, naive attempts to define string diagram languages for non-strict monoidal categories can render the notation so heavy-going as to lose the intuitiveness that makes it so attractive in the first place. 
A more sophisticated solution, which we propose here, is to deliberately use the strictification of a possibly non-strict monoidal category in order to make string diagrams function in this setting with a minimum of additional overhead. 
These points will be illustrated with examples in Section~\ref{section:graphical-language}.

Concretely, the basic idea is to use new operations to `pack' pairs of wires into single wires with internal tensorial structure and to `unpack' structured wires into pairs of wires labelled with the tensor component objects. 
The repeated application of unpacking can flatten any wire with an arbitrarily complex tensor structure into a list of wires labelled with elementary objects. 
Other new operations are used to `hide' or `reveal' wires labelled with the tensor unit. 
These four families of new operations are used to define the associators and the unitors of the strictified category. 

\subsection{Contributions}
We propose a strictification construction yielding a graphical language for
non-strict monoidal categories. With respect to traditional string diagrams, it
provides a more fine-grained representation of tensoring, whose usefulness we
demonstrate in motivating examples drawn from circuit theory and programming
language semantics.
The bulk of the paper is then dedicated to showing that the construction is correct, i.e. the strictified category in which string diagrams live is monoidally equivalent to the original non-strict category. 
Our proof of monoidal equivalence is new: in contrast to Mac Lane's we do not
rely on the coherence theorem, and instead construct the functors of the equivalence
explicitly.
Consequently, we are able to give a new elementary proof of the coherence
theorem: we show \emph{graphically} that the free monoidal category on a single
generator forms a preorder.
The remainder of the coherence result is largely a reformulation Mac Lane's
original corollary, but in a way that we believe has pedagogical value.

\subsection{Related Work}
The use of adapter morphisms which can `pack' and `unpack' wires has been
explored in various forms.
Our adapters can be recovered as instances of more sophisticated constructs used
in the study of coherence of weakly distributive categories
\cite{coherence_weakly_distributive}.
These categories use two distinct tensors and an additional kind of wire
(`thinning links') in their string diagram language.
Note that weakly distributive categories are precisely monoidal categories when the two monoidal structures coincide. This observation is not explored in~\cite{coherence_weakly_distributive} but it is discussed in the context of proof nets in the follow up work~\cite{cockett2017proof}, which also mentions that `thinning links' may be removed when considering the case of monoidal categories. Still, it is
not obvious how this observation would lead to coherence results for
``mere'' monoidal categories as a corresponding simplification of the coherence of
weakly distributive categories.
A non-diagrammatic approach giving a type theory for symmetric monoidal
categories can also be found in~\cite{shulman}.
The idea also appears in the study of quantum~\cite{Carette_2021} and
reversible~\cite{reversible_circuits} circuits,
although these do not study the connection to the strictness and coherence
theorems.
The distinctive feature of our paper is to show explicitly how introducing
adapters allows non-strict categories to take advantage of graph based
datastructures for \emph{strict} monoidal categories such as those of
\cite{wilson2021cost}.
More concretely, by \emph{explicitly} defining the functors mapping between a
non-strict category and its `strictified' counterpart, we obtain
datastructure-independent algorithms for translating between strict and
non-strict settings.
In addition, we formulate our approach using presentations by generators and
relations, which brings the strictness question much closer to current graphical
calculi approaches to circuits such as
\cite{lafont_circuits,polynomial_circuits,full_abstraction,graphical_affine_algebra,passive_linear_networks}.
Finally, our approach allows us to give a self-contained proof of Mac Lane's
strictness theorem that does not rely on the coherence theorem, allowing us to
avoid its associated pitfalls.

Note this paper extends the work published at CSL'23~\cite{DBLP:conf/csl/0002GZ23} with the material presented in Section~\ref{section:symmetric-monoidal-strictness} and Section~\ref{section:corollary-proof}. 

\subsection{Synopsis}
In Section~\ref{section:graphical-language} we present our graphical calculus
for (non-strict) monoidal categories, in the form of a strictification
procedure.
Subsections~\ref{section:motivating-examplesI}, \ref{section:exstlc}, and
\ref{section:exnonstrict} illustrate a series of motivating examples.
Section~\ref{section:strictness} justifies our construction by proving that it
yields an equivalence of categories.
Section~\ref{section:symmetric-monoidal-strictness} shows how the strictness
result extends to the symmetric monoidal case.
Section~\ref{section:coherence} revisits MacLane's Coherence theorem and some of
its consequences in light of the approach we presented.
Finally, in Section~\ref{section:corollary-proof}, we give a graphical
exposition of Mac Lane's \emph{corollary} to the coherence theorem.
Section~\ref{section:conclusions} is dedicated to conclusions and future work.

\section{A graphical language for (non-strict) monoidal categories}
\label{section:graphical-language}

We assume familiarity with string diagrams for strict monoidal categories, see e.g.~\cite{selinger2010survey}. Let us fix an arbitrary (non-strict) monoidal category $\CatC$. We construct its \emph{strictification} as the strict monoidal category $\CatD$ defined as follows.
\begin{definition}
  \label{definition:catd}
  $(\CatD, \TensorD)$ is the strict monoidal category freely generated by:
  \begin{enumerate}
    \item Objects $\ObjD{A}$ for each $A \in \CatC$
    \item Generators \eqref{equation:catd-generators}, with $\MorD{f} : \ObjD{A} \to \ObjD{B}$ for each $f : A \to B \in \CatC$
    \item functoriality equations \eqref{equation:catd-functor-equations}
    \item adapter equations \eqref{equation:catd-adapter-equations}, and
    \item associator/unitor equations \eqref{equation:catd-monoidal-equations}
  \end{enumerate}
\end{definition}

Note that because $\CatD$ is strict by definition, we are entitled to use string
diagrammatic notation.
Thus, a morphism with $m$ inputs and $n$ outputs will have domain and codomain
of the form
$\ObjD{A_1} \TensorD \cdots \TensorD \ObjD{A_m}$
and $\ObjD{B_1} \TensorD \cdots \TensorD \ObjD{B_n}$,
respectively.

\vbox{
\begin{multicols}{2}
  \begin{equation}
    \label{equation:catd-generators}
    \begin{gathered}
      \stikzfig{generator-big-adapter} \qquad \stikzfig{generator-big-adapterInv} \\
      \stikzfig{generator-adapter} \qquad \qquad \stikzfig{generator-adapterInv} \\
      \stikzfig{generator-strictf}
    \end{gathered}
  \end{equation}
  \dotfill
  \begin{equation}
    \label{equation:catd-functor-equations}
    \begin{gathered}
      \begin{aligned}
        \stikzfig{equation-identity-lhs}    \quad & = \quad \stikzfig{equation-identity-rhs} \\
        \stikzfig{equation-composition-lhs} \quad & = \quad \stikzfig{equation-composition-rhs}
      \end{aligned}
    \end{gathered}
  \end{equation}
  \break
  \begin{gather}
    \label{equation:catd-adapter-equations}
      \begin{aligned}
        \stikzfig{equation-Adapter-naturality-1-lhs} \quad & = \quad \stikzfig{equation-Adapter-naturality-1-rhs}
        \\
        \stikzfig{equation-Adapter-naturality-2-lhs} \quad & = \quad \stikzfig{equation-Adapter-naturality-2-rhs}
        \\
        \stikzfig{equation-adapter-isomorphism-1-lhs} \quad & = \quad \stikzfig{empty}
        \\
        \stikzfig{equation-adapter-isomorphism-2-lhs} \quad & = \quad \stikzfig{equation-adapter-isomorphism-2-rhs}
      \end{aligned}
  \end{gather}
\end{multicols}
}
\begin{gather}
  \label{equation:catd-monoidal-equations}
    \begin{aligned}
      \MorD{\assoc}    \quad & = \quad \stikzfig{equation-associator-rhs} \qquad
      & \MorD{\assocInv} \quad & = \quad \stikzfig{equation-associatorInv-rhs} \\
      \MorD{\unitl}    \quad & = \quad \stikzfig{equation-unitl-rhs} \qquad
      & \MorD{\unitlInv} \quad & = \quad \stikzfig{equation-unitlInv-rhs} \\
      \MorD{\unitr}    \quad & = \quad \stikzfig{equation-unitr-rhs} \qquad
      & \MorD{\unitrInv} \quad & = \quad \stikzfig{equation-unitrInv-rhs}
    \end{aligned}
\end{gather}
This is a functorial construction, yielding a monoidal equivalence between
$\CatC$ and $\CatD$, as we will prove in Section~\ref{section:strictness}.
Note that although the category $\CatD$ is essentially the same as that given by
Mac Lane~\cite[p.~257]{CFWM}, its construction differs in one key respect.
Namely, to define his equivalent strict category, Mac Lane relies on the
coherence theorem to define both composition of arrows and to ensure the
functors in the equivalence are monoidal.
In contrast, the adapter generators and equations of $\CatD$ mean that
Definition \ref{definition:catd} does not require use of the coherence theorem,
and can therefore be used to prove it.

The functoriality equations are so-called as they ensure functoriality of the
construction. The `adapter' equations and `associator/unitor' equations further
ensure this functor is \emph{monoidal} and it forms one half of a
\emph{monoidal equivalence}.
Sec.~\ref{section:strictness} will make it clear that these equations
are essentially obtained by freely adding the morphisms required by the
definition of a monoidal functor (\ref{definition:monoidal-functor}).

Besides its mathematical significance, the interest of this construction lies in
providing a means of manipulating morphisms of non-strict monoidal categories
graphically.
In particular, the $\adapter$ and $\adapterInv$ generators can be used to
explicitly summon and dispell the monoidal unit, while the $\bigadapter$ and
$\bigadapterInv$ generators can be thought of as systematic ways of packing and
unpacking wires into more complex wires with internal structure.
The next subsections will showcase how this additional layer of structure can be
useful in categorical models of computation.

\subsection{Circuit Description Languages with Tuples}
\label{section:motivating-examplesI}

Categorical models of circuit description languages are a prime source of examples of monoidal categories, for instance combinational~\cite{lafont_circuits} or sequential~\cite{DBLP:conf/csl/GhicaJL17} circuits. 
The graphical representation of circuits also fits naturally and intuitively the box-and-wire model used by string diagrams. 
More precisely, the circuit description languages in \emph{loc. cit.} (and variations thereof) are instances of \emph{strict} monoidal categories. 

From the point of view of \emph{expressiveness}, i.e. realising circuits with certain desired behaviours, the strict setting does not introduce any limitations. 
Consistent with this observation, standard hardware description languages (HDL) such as Verilog can also be modelled using a strict monoidal tensor. 
However, larger and more complex designs stand to benefit from the additional level of structure which a non-strict tensor can offer and, indeed, more modern HDLs, intended for more complex designs, such as SystemVerilog have syntactic facilities which require a non-strict tensor: \emph{struct}s. 

Consider the following simple example. 
Suppose that some circuitry is needed to process network packets, which consist of a header (of size $h=96$ bits), a payload (of size $p=896$ bits) and an error-correcting trailer (of size $e=32$ bits). 
In the older Verilog language, the header and the payload can be combined in a single, wider, data bus of $h+p=992$ bits, but the two components can only be extracted using numerical indexing. 
This is a primitive form of `flattening' a data structure into an array, and in the more modern SystemVerilog it can be avoided by using a \emph{struct}. 
This means that a data type of `message' (say $m$) can access its components as \emph{fields} (projections), namely $m.h$ and $m.p$. 
Since structs can have other structs as fields, the way in which the components are associated is relevant, which means that the tensor must no longer be strict. 

On the other hand, `flattening' the structure of a data bus to an array of bits can be useful. 
In the current example, in computing the error-correcting code $e$, the way the message is partitioned into header and payload is no longer relevant, so it is convenient to unpack the tensor $h\otimes p$ into a flat array of $h+p$ wires from which an error-correcting code $e$ is computed by a generic circuit of the appropriate width. 
Structures that can be flattened like this are called in SystemVerilog \emph{packed structs}, and to model them properly both strict and non-strict tensorial facilities are required in the categorical model. 

Finally, the error-correcting code can be packed with the original message into an error-correcting message with three components. 
It is obviously important to be able to retrieve the header, payload, and error-correcting code separately from the message, and it should be equally obvious that once the internal structure of the message is non-trivial a calculus of indices would be a complicated, awkward, and error-prone way to access the components. 
\[
  \centerline{\stikzfig{case-study-parity-3}}
\]
Graphically, this circuit is represented above.
In order to make this diagram completely formal, what we are using here is the $\CatD$ construction described in Sec.~\ref{section:graphical-language} applied to one of the categories $\CatC$ of digital circuits (combinational or sequential) mentioned earlier. 
This gives us the best of both worlds: the `non-strictness' of
circuits-with-tuples, and the graphical syntax of string diagrams.

\paragraph{Strictifying Strict Categories}
The `strictification' procedure is not just useful for
providing a graphical syntax for non-strict monoidal categories, but can also
provide a more ergonomic syntax for monoidal categories that are \emph{already} strict.
Suppose we wish to work in Lafont's strict monoidal category of
circuits~\cite{lafont_circuits}, and suppose we would like to define the `parity'
function used earlier.
Using our construction, we can define it recursively as follows:
\[
  \mathsf{parity}_0 = \stikzfig{case-study-recursive-definition-parity-basecase}
  \qquad \qquad
  \mathsf{parity}_n = \stikzfig{case-study-recursive-definition-parity}
\]
Notice that in the `base' language of Lafont's PROP of circuits we cannot truly
depict this diagram, since there is no way to treat a bundle of $n$ wires as a
pair of $1$ and $n - 1$ wires.
To do this formally we require the adapter morphisms as defined in Section
\ref{section:graphical-language}.

\subsection{Programming Languages}
\label{section:exstlc}
Programming languages, largely based on the lambda calculus, commonly include product formation as a syntactic feature. 
Therefore, a graphical syntax based on its categorical model, as used for example in~\cite{DBLP:journals/corr/abs-2107-13433}, needs to have a non-strict tensor.
However, having only the non-strict tensor leads to an awkward graphical syntax in which all generators have a single wire going in and a single wire going out. 
Diagrams in which the interfaces can be intermediated using lists of wires require mechanisms for strictification. 
This can be realised by applying the strictification construction to a Cartesian closed category, which will allow the expression of examples such as the one below. 

\topicsentence{
  Consider the simple task of summing two complex numbers,
  whose real and imaginary parts are encoded as floating-point numbers.
}
That is, while we have a primitive type of reals, we model complex numbers as
pairs $\Complex = \Real \times \Real$.
A natural way to write such a program in a diagrammatic form is pictured below.
\[
  \centerline{\stikzfig{case-study-stlc}}
\]
Even in categorical models of the simply-typed $\lambda$-calculus (STLC) without product, strictification has a role to play.
As usual, this role is cloaked in informality which in some contexts can lead to ambiguity. 
STLC is interpreted by giving meaning to type judgements $\Gamma\vdash t:T$ with $\Gamma$ a context, $t$ a term, and $T$ a type. 
The context $\Gamma=x_1:T_1,\ldots,x_n:T_n$ is a list of typed variables which is interpreted as the tensor $T_1\otimes\cdots\otimes T_n$, virtually always treated as if it were strict.
This informal strictification can be problematic though when product types are used, as the objects $T_i$ in the interpretation of the context also contain tensors.
So the strictification must be fine-grained enough to allow only the flattening of those tensors representing the comma of the context, and not those of the product formation.
Our approach offers this level of granularity. 

\subsection{Strict vs. Non-Strict String Diagrams}
\label{section:exnonstrict}
Our final example concerns the usability problems of non-strict diagrams \emph{without} strictification and illustrate how our approach to strictification with packing and unpacking wires makes rigorous the intuition that formulating certain properties in terms of strict monoidal categories does not entail a loss of generality. 
Our example uses
braided autonomous categories. Here, each object $A$ has a dual $A^*$, there exists a family of isomorphisms $c_{A,B}:A\otimes B\to B\otimes A$ called braidings, and families of adjunctions $\eta_A:I\to A^*\otimes A$, $\epsilon:A\otimes A^*\to I$ with certain properties which we may elide in the formulation of the example.
Consider the property of braided monoidal categories to be autonomous if and only if they are right-autonomous~\cite[Prop. 7.2]{JOYAL199320}. 
The proof is formulated in terms of string diagrams in~\cite[Lem. 4.17]{selinger2010survey}, which makes it more intuitive. 

The idea of the proof is to show that isomorphisms $b_A:A^{**}\to A$, $b_A^{-1}:A \to A^{**}$ can be constructed. They are defined as follows: 
\begin{align*}
& b_A = A^{**} \xrightarrow{\eta_A\otimes \id}A^*{\otimes} A{\otimes} A^{**}
\xrightarrow{\id\otimes c_{A,A^{**}}} A^*{\otimes} A^{**}{\otimes} A
\xrightarrow{\epsilon_{A^*}\otimes \id}A\\
& b_A^{-1} = A \xrightarrow{\id\otimes\eta_{A^*}}A{\otimes} A^{**}{\otimes} A^*
\xrightarrow{c^{-1}_{A^{**}, A}} A^{**}{\otimes} A{\otimes} A^*
\xrightarrow{\id\otimes\epsilon_{A}}A^{**}.
\end{align*}
The fact that $b_A;b_A^{-1}=\id$ becomes elegantly obvious when the terms are rendered as string diagrams which can be manipulated graphically:
\[ \centerline{\stikzfig{motivating-examples-lemma-1}} \]
The exposition includes the standard caveat that ``Here we have written, without loss of generality, as if [the category] were strict monoidal.''
We shall now show, graphically, that this is indeed the case.

First we note that in the non-strict setting (without strictification) all string diagrams must be equipped with gadgets that make sure that there is a single wire on the left, and a single wire on the right. 
These gadgets are of course the bundlers and unbundlers introduced earlier. 
Therefore, in the non-strict setting, taking into account all the relevant
associators, the diagram for $b_A$ becomes much more complicated, denying the
intuitiveness we expect from a graphical notation (see below).
\[
  \centerline{\scalebox{0.7}{
\begingroup%
  \makeatletter%
  \providecommand\color[2][]{%
    \errmessage{(Inkscape) Color is used for the text in Inkscape, but the package 'color.sty' is not loaded}%
    \renewcommand\color[2][]{}%
  }%
  \providecommand\transparent[1]{%
    \errmessage{(Inkscape) Transparency is used (non-zero) for the text in Inkscape, but the package 'transparent.sty' is not loaded}%
    \renewcommand\transparent[1]{}%
  }%
  \providecommand\rotatebox[2]{#2}%
  \newcommand*\fsize{\dimexpr\f@size pt\relax}%
  \newcommand*\lineheight[1]{\fontsize{\fsize}{#1\fsize}\selectfont}%
  \ifx\svgwidth\undefined%
    \setlength{\unitlength}{332.68763339bp}%
    \ifx\svgscale\undefined%
      \relax%
    \else%
      \setlength{\unitlength}{\unitlength * \real{\svgscale}}%
    \fi%
  \else%
    \setlength{\unitlength}{\svgwidth}%
  \fi%
  \global\let\svgwidth\undefined%
  \global\let\svgscale\undefined%
  \makeatother%
  \begin{picture}(1,0.12887253)%
    \lineheight{1}%
    \setlength\tabcolsep{0pt}%
    \put(0,0){\includegraphics[width=\unitlength,page=1]{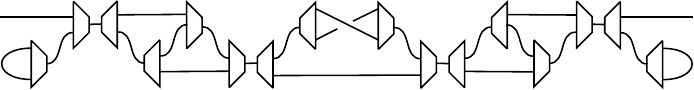}}%
    \put(0.10755479,0.08814166){\color[rgb]{0,0,0}\makebox(0,0)[lt]{\lineheight{1.25}\smash{\begin{tabular}[t]{l}$1$\end{tabular}}}}%
    \put(0.14833362,0.08858006){\color[rgb]{0,0,0}\makebox(0,0)[lt]{\lineheight{1.25}\smash{\begin{tabular}[t]{l}$1$\end{tabular}}}}%
    \put(0.04873201,0.02866706){\color[rgb]{0,0,0}\makebox(0,0)[lt]{\lineheight{1.25}\smash{\begin{tabular}[t]{l}$2$\end{tabular}}}}%
    \put(0.21186202,0.03174496){\color[rgb]{0,0,0}\makebox(0,0)[lt]{\lineheight{1.25}\smash{\begin{tabular}[t]{l}$2$\end{tabular}}}}%
    \put(0.33344002,0.03020601){\color[rgb]{0,0,0}\makebox(0,0)[lt]{\lineheight{1.25}\smash{\begin{tabular}[t]{l}$3$\end{tabular}}}}%
    \put(0.3757615,0.03020601){\color[rgb]{0,0,0}\makebox(0,0)[lt]{\lineheight{1.25}\smash{\begin{tabular}[t]{l}$3$\end{tabular}}}}%
    \put(0.27342047,0.08791713){\color[rgb]{0,0,0}\makebox(0,0)[lt]{\lineheight{1.25}\smash{\begin{tabular}[t]{l}$4$\end{tabular}}}}%
    \put(0.43655048,0.08714772){\color[rgb]{0,0,0}\makebox(0,0)[lt]{\lineheight{1.25}\smash{\begin{tabular}[t]{l}$4$\end{tabular}}}}%
    \put(0.83591123,0.08791713){\color[rgb]{0,0,0}\makebox(0,0)[lt]{\lineheight{1.25}\smash{\begin{tabular}[t]{l}$5$\end{tabular}}}}%
    \put(0.87669373,0.08791713){\color[rgb]{0,0,0}\makebox(0,0)[lt]{\lineheight{1.25}\smash{\begin{tabular}[t]{l}$5$\end{tabular}}}}%
    \put(0.77435274,0.03020601){\color[rgb]{0,0,0}\makebox(0,0)[lt]{\lineheight{1.25}\smash{\begin{tabular}[t]{l}$6$\end{tabular}}}}%
    \put(0.93671334,0.02943647){\color[rgb]{0,0,0}\makebox(0,0)[lt]{\lineheight{1.25}\smash{\begin{tabular}[t]{l}$6$\end{tabular}}}}%
    \put(0.61045325,0.03020601){\color[rgb]{0,0,0}\makebox(0,0)[lt]{\lineheight{1.25}\smash{\begin{tabular}[t]{l}$7$\end{tabular}}}}%
    \put(0.65123569,0.03020601){\color[rgb]{0,0,0}\makebox(0,0)[lt]{\lineheight{1.25}\smash{\begin{tabular}[t]{l}$7$\end{tabular}}}}%
    \put(0.5488947,0.08714759){\color[rgb]{0,0,0}\makebox(0,0)[lt]{\lineheight{1.25}\smash{\begin{tabular}[t]{l}$8$\end{tabular}}}}%
    \put(0.71279425,0.08714759){\color[rgb]{0,0,0}\makebox(0,0)[lt]{\lineheight{1.25}\smash{\begin{tabular}[t]{l}$8$\end{tabular}}}}%
  \end{picture}%
\endgroup%
}}
\]
This is why a naive approach to non-strict string diagram construction is not effective. 
However, the complications are only an artefact of the construction of the diagram in a purely non-strict setting.
The strictification equations come to rescue and, in this case, cancel out all
bundler-unbundler pairs in the order indicated by the numerical labels attached
to them, resulting in exactly the same diagram of $b_A$ that was constructed in
the strict setting.
So, indeed, working in the strict setting implied no loss of generality!

\section{Strictness}
\label{section:strictness}
\topicsentence{
  We now show that $\CatC$ is monoidally equivalent to $\CatD$, constituting a
  proof of Mac Lane's strictness theorem, since $\CatC$ is an arbitrary monoidal
  category.
}
Our approach is to define monoidal functors $\FunF : \CatC \to \CatD : \FunG$,
and we begin by recalling the definition of monoidal functor.

\begin{definition} \deftext{Monoidal Functor} \\
  \label{definition:monoidal-functor}
  Let $(\cat{C}, \otimes, \unit_\cat{C})$
  and $(\cat{D}, \bullet, \unit_\cat{D})$
  be monoidal categories.
  A \textit{monoidal functor} is a functor $F : \cat{C} \to \cat{D}$
  equipped with natural isomorphisms
  $\bigadapter_{X,Y} : F(X) \bullet F(Y) \to F(X \otimes Y)$
  and
  $\phi : \unit_\cat{D} \to F(\unit_\cat{C})$
  such that the following diagrams commute for all objects $A, B, C \in \cat{C}$.
  
  \begin{equation}
    \label{equation:monoidal-hexagon}
    \scalebox{0.7}{\begin{tikzpicture}[node distance=2cm, baseline=(current  bounding  box.center)]
  \node (A) [] {$(F(A) \bullet F(B)) \bullet F(C)$};
  \node (B) [right of=A,xshift=3cm] {$F(A) \bullet (F(B) \bullet F(C))$};
  \node (C) [below of=B] {$F(A) \bullet F(B \otimes C)$};
  \node (D) [below of=C] {$F(A \otimes (B \otimes C))$};

  \node (F) [below of=A] {$F(A \otimes B) \bullet F(C)$};
  \node (E) [below of=F] {$F((A \otimes B) \otimes C)$};

  \draw [->,thick] (B) to [] node[above] {$\assoc_\cat{D}$} (A);
  \draw [->,thick] (B) to [] node[right] {$\id_{F(A)} \bullet \Phi_{B,C}$} (C);
  \draw [->,thick] (C) to [] node[right] {$\Phi_{A,B \otimes C}$} (D);

  \draw [->,thick] (A) to [] node[left] {$\Phi_{A,B} \bullet \id_{F(C)}$} (F);
  \draw [->,thick] (F) to [] node[left] {$\Phi_{A \otimes B, C}$} (E);
  \draw [->,thick] (D) to [] node[below] {$F(\assoc_\cat{C})$} (E);
\end{tikzpicture}}
  \end{equation}

  \begin{equation}
    \label{equation:monoidal-squares}
    \scalebox{0.7}{\begin{tikzpicture}[node distance=3.25cm, baseline=(current  bounding  box.center)]
  \node (A) [] {$F(A) \bullet \unit_\cat{D}$};
  \node (B) [right of=A] {$F(A) \bullet F(\unit_\cat{C})$};
  \node (C) [below of=B] {$F(A \otimes \unit_\cat{C})$};
  \node (D) [left  of=C] {$F(A)$};

  \draw [->,thick] (A) to [] node[above] {$\id_{F(A)} \bullet \phi$} (B);
  \draw [->,thick] (B) to [] node[right] {$\Phi_{A, \unit_\cat{C}}$} (C);
  \draw [->,thick] (C) to [] node[below] {$F(\rho_\cat{C})$} (D);
  \draw [->,thick] (A) to [] node[left ] {$\rho_{\cat{D}}$} (D);
\end{tikzpicture}
\enspace
\begin{tikzpicture}[node distance=3.25cm, baseline=(current  bounding  box.center)]

  \node (A) []           {$\unit_\cat{D} \bullet F(B)$};
  \node (B) [right of=A] {$F(\unit_\cat{C}) \bullet F(B)$};
  \node (C) [below of=B] {$F(\unit_\cat{C} \otimes B)$};
  \node (D) [left  of=C] {$F(B)$};

  \draw [->,thick] (A) to [] node[above] {$\phi \bullet \id_{F(B)}$} (B);
  \draw [->,thick] (B) to [] node[right] {$\Phi_{\unit_\cat{C}, B}$} (C);
  \draw [->,thick] (C) to [] node[below] {$F(\lambda_\cat{C})$} (D);
  \draw [->,thick] (A) to [] node[left ] {$\lambda_\cat{D}$} (D);
\end{tikzpicture}}
  \end{equation}

\end{definition}

\topicsentence{With this definition it is straightforward to see how to define a
monoidal functor from $\CatC$ to $\CatD$.}

\begin{definition}
  Let $\FunF : \CatC \to \CatD$ be the \emph{strictification functor}
  defined on objects and morphisms as
  $\FunF(A) \defeq{} \ObjD{A}$
  and
  $\FunF(f) \defeq{} \MorD{f}$,
  respectively
\end{definition}

\begin{proposition} $(\FunF, \bigadapter, \adapter)$ is a monoidal functor.
  \label{proposition:funf-is-monoidal}
\end{proposition}

\begin{proof}
  $\FunF$ preserves identities and composition (and is therefore a functor) by
  the functor equations \eqref{equation:catd-functor-equations}:
  \[
    \FunF(\id_A) = \MorD{\id_A} = \id_{\ObjD{A}}
    \qquad
    \qquad
    \qquad
    \FunF(f \cp g) = \MorD{f \cp g} = \MorD{f} \cp \MorD{g} = \FunF(f) \cp \FunF(g)
  \]
  It is a \emph{monoidal} functor using the adapter generators
  $\bigadapter = \generator{g-big-adapter}$ and
  $\adapter = \generator{g-adapter}$
  from \eqref{equation:catd-generators}.
  For this to work, we must have that $\generator{g-big-adapter}$ is a natural isomorphism
  and $\generator{g-adapter}$ an isomorphism, respectively.
  This is a straightforward consequence of the adapter equations
  \eqref{equation:catd-adapter-equations}:
  $\bigadapterInv \cp (\MorD{f} \TensorD \MorD{g})
      = \bigadapterInv \cp (\MorD{f} \TensorD \MorD{g}) \cp \bigadapter \cp \bigadapterInv
      = \MorD{f \TensorC g} \cp \bigadapterInv$
  and $\phi \cp \phi^* = \id$ by definition.
  Similarly, we require that the diagrams of \eqref{equation:monoidal-hexagon}
  and \eqref{equation:monoidal-squares} commute.
  Again, this is precisely what the the associator/unitor equations
  \eqref{equation:catd-monoidal-equations} state, and so $\FunF$ is a monoidal
  functor.
\end{proof}

\begin{remark}
  Notice that $\CatD$ is \emph{defined} by freely adding the requirements of
  Definition \ref{definition:monoidal-functor}.
  Generators $\generator{g-big-adapter}$ and $\generator{g-adapter}$ and equations
  \eqref{equation:catd-adapter-equations} give the natural isomorphism
  $\bigadapter$ and isomorphism $\adapter$,
  while the commuting diagrams \eqref{equation:monoidal-hexagon} and
  \eqref{equation:monoidal-squares} are precisely the
  `associator/unitor' equations \eqref{equation:catd-monoidal-equations}.
\end{remark}

\topicsentence{
  We can now define the other half of the monoidal equivalence $\FunF \dashv
  \FunG$.
}
In doing so, we'll make use of the fact that morphisms of a monoidal category
can be written in a `sequential normal form'
(Appendix \ref{section:sequential-normal-form}), i.e. as a series of `slices'
\[ (\id \otimes g_1 \otimes \id) \cp (\id \otimes g_2 \otimes \id) \cp \ldots \cp (\id \otimes g_n \otimes \id) \]
where each $g_i$ is a generator.
We take advantage of this form to define $\FunG$: our definition is defined on
`slices' $\id_X \TensorD q \TensorD \id_Y$ for some generator $q$,
and then freely on composition so that $\FunG(f \cp g) = \FunG(f) \cp \FunG(g)$.

\begin{definition}
  \label{definition:fung}
  We define the nonstrictification functor
  $\FunG : \CatD \to \CatC$ inductively on objects:
  \[
    \FunG(\unitD)              \defeq \unitC
    \qquad \qquad
    \FunG(\ObjD{A})            \defeq A
    \qquad \qquad
    \FunG(\ObjD{A} \TensorD R) \defeq A \TensorC \FunG(R)
  \]
  And on morphisms we give a recursive definition, with the following base cases:
  \begin{equation*}
    \begin{aligned}
      \FunG(\id_{\unitD})                       & \defeq \id_{\unitC} \\
      \FunG(\MorD{f})                           & \defeq f \\
      \FunG(\bigadapter_{A,B})                     & \defeq \id_{A \otimes B}    = \FunG(\bigadapterInv_{A,B}) \\
      \FunG(\adapter)                           & \defeq \id_{\unitC} = \FunG(\adapterInv)
    \end{aligned}
    \quad
    \begin{aligned}
      \FunG(\MorD{f} \TensorD \id_Y)            & \defeq f \TensorC \id_{\FunG(Y)} \\
      \FunG(\bigadapter_{A,B} \TensorD \id_Y)      & \defeq \assoc_{A,B, \FunG(Y)} \\
      \FunG(\bigadapterInv_{A,B} \TensorD \id_Y)   & \defeq \assocInv_{A,B, \FunG(Y)} \\
      \FunG(\adapter \TensorD \id_Y)            & \defeq \unitlInv_{\FunG(Y)} \\
      \FunG(\adapterInv \TensorD \id_Y)         & \defeq  \unitl_{\FunG(Y)}
    \end{aligned}
    \quad
    \begin{aligned}
      \FunG(\id_{\ObjD{A}} \TensorD \MorD{f})           & \defeq \id_A \TensorC f \\
      \FunG(\id_{\ObjD{A}} \TensorD \bigadapter_{B,C})     & \defeq \id_{A \TensorC (B \TensorC C)} \\
      \FunG(\id_{\ObjD{A}} \TensorD \bigadapterInv_{B,C} ) & \defeq \id_{A \TensorC (B \TensorC C)} \\
      \FunG(\id_{\ObjD{A}} \TensorD \adapter)           & \defeq \unitrInv_{A} \\
      \FunG(\id_{\ObjD{A}} \TensorD \adapterInv)        & \defeq \unitr_{A}
    \end{aligned}
  \end{equation*}
  With a single recursive case, for
  $q \in \{ \bigadapter, \adapter, \bigadapterInv, \adapterInv, \id_{\ObjD{Q}} \}$
  \begin{align*}
    \FunG(\id_{\ObjD{A}} \TensorD q \TensorD r) \defeq \id_A \TensorC \FunG(q \TensorD r)
  \end{align*}
  Finally take $\FunG(f \cp g) \defeq \FunG(f) \cp \FunG(g)$.
\end{definition}

This definition is well defined with respect to the equations of Definition
\ref{definition:catd};
we give a proof in Appendix \ref{section:g-well-defined},
where we also note that $N(f)$ is the same regardless of which `sequential normal form'
decomposition we choose for $f$.

\begin{remark}
  The definition of $\FunG$ can be explained more intuitively in terms of
  programming.
  If we think of each `slice' of the sequential normal form as a list of
  primitive arrows of $\CatC$,
  then the definition of $\FunG$ is essentially a list recursion in which
  we have a separate case for $1$, $2$, and $n$-element lists.
\end{remark}



\topicsentence{Now we will show that $\FunG$ is a \emph{monoidal} functor.}
To do this, we must specify the `coherence maps': a natural isomorphism
$\Psi_{X,Y} : \FunG(X) \TensorC \FunG(Y) \to \FunG(X \TensorD Y)$
and isomorphism
$\psi : \unitC \to \FunG(\unitD)$
as mandated by Definition \ref{definition:monoidal-functor}.

\begin{definition} 
  We define $\Psi$, the coherence natural isomorphism for $\FunG$, in the following cases:
  \[
    \Psi_{\unitD, \unitD}
        \defeq \unitl_{\unitC} = \unitr_{\unitC}
    \qquad \qquad
    \Psi_{X, \unitD}
        \defeq \unitr_{\FunG(X)}
    \qquad \qquad
    \Psi_{\unitD, Y}
        \defeq \unitl_{\FunG(Y)}
  \]
  \[
    \Psi_{\ObjD{A}, Y}
        \defeq \id_{A \TensorC \FunG(Y)}
    \qquad \qquad
    \Psi_{\ObjD{A} \TensorD X, Y}
        \defeq \assocInv_{A, \FunG(X), \FunG(Y)} \cp (\id_{A} \TensorC \Psi_{X, Y})
  \]
\end{definition}

\begin{definition} 
  The coherence isomorphism $\psi$ for $\FunG$ is defined as follows:
  \begin{align*}
    \psi_{\unitC}
        & \defeq \id_{\unitC}
  \end{align*}
\end{definition}

\begin{remark}
  Note that both $\unitl_{\unitC}$ and $\unitr_{\unitC}$ have the correct type
  as a choice for $\Psi_{\unitD, \unitD}$.
  In fact, they are equal: unitors coincide at the unit object, i.e.
  $\unitl_{\unitC} = \unitr_{\unitC}$, as noted
  in~\cite[Corollary 2.2.5]{tensor_categories}.
\end{remark}

\begin{proposition} $(\FunG, \Psi, \psi)$ is a monoidal functor.
  \label{proposition:fung-is-monoidal}
\end{proposition}

\begin{proof}
  It is clear that $\Psi$ and $\psi$ are natural isomorphisms since they are
  both composites of natural isomorphisms.
  Thus it remains to check the diagrams of Definition
  \ref{definition:monoidal-functor} commute.

  The squares \eqref{equation:monoidal-squares} commute because $\psi = \id$,
  and $\Psi_{A,\unitD} = \unitr$
  and $\Psi_{\unitD,B} = \unitl$
  by definition.

  Now let us check that the hexagon \eqref{equation:monoidal-hexagon} commutes.
  Note that in the following we use that $\FunG(\assocD) = \id$, because $\CatD$
  is strict, and so the hexagon axiom becomes a pentagon.

  We will approach the problem inductively, checking base cases where
  $A = \unit$ and $A = \ObjD{A}$, and finally the inductive step with
  $A = \ObjD{A} \TensorD R$.
  Let us begin with $A = \unit$, and taking the outer path of the hexagon we
  calculate as follows:
  \begin{align*}
      (\id_{\unitC} \TensorC \Psi_{B,C})
    & \cp \Psi_{\unitC, B \TensorD C}
      \cp \Psi^{-1}_{B,C}
      \cp (\Psi_{\unitC, B} \TensorC \id_{\FunG(C)})^{-1} \\
    & = (\id_{\unitC} \TensorC \Psi_{B,C})
      \cp \unitl_{\FunG(B \TensorD C)}
      \cp \Psi^{-1}_{B,C}
      \cp (\unitl_{\FunG(B)} \TensorC \id_{\FunG(C)})^{-1} \\
    & = \unitl_{\FunG(B) \TensorC \FunG(C)}
      \cp (\unitl_{\FunG(B)} \TensorC \id_{\FunG(C)})^{-1} \\
    & = \assoc_{\unitC, \FunG(B), \FunG(C)}
  \end{align*}
  Wherein we expanded the definition of $\Psi$, then used naturality of $\Psi_{B,C}$
  before applying the monoidal triangle lemma of~\cite[(2.12)]{tensor_categories}.

  Now consider the second base case, where $A$ is the `singleton list' $\ObjD{A}$.
  In this case, the hexagon diagram commutes immediately
  because $\Psi_{\ObjD{A},B} = \id_{\ObjD{A} \TensorC \FunG(B)}$
  and $\Psi_{\ObjD{A}, B \TensorD C} = \id_{\ObjD{A} \TensorC \FunG(B \TensorD C)}$.
  More explicitly, we calculate as follows, starting again with the outer path
  of the hexagon and expanding definitions:
  \begin{align*}
    (\id_A \TensorC \Psi_{B,C})
    & \cp \Psi_{A, B \TensorD C}
      \cp \Psi^{-1}_{A \TensorD B, C}
      \cp (\Psi_{\ObjD{A},B} \TensorC \id_{\FunG(C)}) \\
    & = (\id_A \TensorC \Psi_{B,C})
      \cp (\id_A \TensorC \Psi_{B,C})^{-1}
      \cp \assoc_{A, \FunG(B), \FunG(C)} \\
    & = \assoc_{A, \FunG(B), \FunG(C)}
  \end{align*}
  Finally let us prove the inductive step.
  Assume that the hexagon commutes for objects $R$, $B$, $C$, giving us the
  equation
  \[
    \Psi_{R, B \TensorD C} \cp \Psi^{-1}_{R \TensorD B, C}
      = (\id_{\FunG(R)} \TensorC \Psi^{-1}_{B,C})
      \cp \assoc_{\FunG(R), \FunG(B), \FunG(C)}
      \cp (\Psi_{R,B} \TensorC \id_{\FunG(C)})
  \]
  We may then rewrite the following subterm of the monoidal hexagon as follows:
  \[
    \id_A \TensorC (\Psi_{R, B \TensorD C} \cp \Psi^{-1}_{R \TensorD B, C})
      =   \id_A \TensorC (\id_{\FunG(R)} \TensorC \Psi^{-1}_{B,C})
      \cp \id_A \TensorC \assoc_{\FunG(R), \FunG(B), \FunG(C)}
      \cp \id_A \TensorC (\Psi_{R,B} \TensorC \id_{\FunG(C)})
  \]
  We can then rewrite
  $\id_A \TensorC \assoc_{\FunG(R), \FunG(B) \FunG(C)}$
  using the monoidal category pentagon axiom,
  and then use naturality of $\assoc$ to reduce the outer path of the
  monoidal hexagon until we are left with
  $\assoc_{A \TensorC \FunG(R), \FunG(B), \FunG(C)}$,
  as required.
\end{proof}

\topicsentence{
  Finally, we must check that $\FunF$ and $\FunG$ indeed form an equivalence.
}
First, recall the definition

\begin{definition}[Equivalence of categories]
  An equivalence is a pair of functors
  $ \cat{C} \smash{\overset{F}{\rightarrow} \atop \underset{G}{\leftarrow}} \cat{D} $
  and a pair of natural isomorphisms
  $\eta : \id_{\cat{C}} \to G \circ F$
  and
  $\epsilon : F \circ G \to \id_{\cat{D}}$.
\end{definition}


\topicsentence{We begin by showing naturality of $\eta$.}

\begin{proposition} $\FunG \circ \FunF = \id_{\CatC}$.
  \label{proposition:f-then-g-is-identity}
\end{proposition}

\begin{proof} $\FunG(\FunF(f)) = \FunG(\MorD{f}) = f = \id_{\CatD}(f)$
\end{proof}

\begin{remark}
  Note that Proposition \ref{proposition:f-then-g-is-identity} shows that
  the composite $\FunG \circ \FunF$ is actually \emph{equal} to the identity
  functor, and thus $\eta_A = \id_A$.
  In fact, this will make the composite of the two functors a
  \emph{split idempotent}.
\end{remark}

\topicsentence{Now we prove naturality of $\epsilon$.}
This proof is somewhat more involved: unlike
\ref{proposition:f-then-g-is-identity}, the composite $\FunF \circ \FunG$ is
merely isomorphic to the identity functor, not equal on the nose.
Thus, we begin with an inductive definition:

\begin{definition} 
  We define the (monoidal) natural isomorphism
  $\epsilon : \FunF \circ \FunG \to \id_{\CatD}$
  for the composite $\FunF \circ \FunG$ inductively:
  \label{definition:epsilon}
  \begin{equation}
    \begin{aligned}
      \epsilon_{\unitD}               & \defeq \adapterInv    & = & \stikzfig{generator-adapterInv-bare} \\
      \epsilon_{\ObjD{A}}             & \defeq \id_{\ObjD{A}} & = & \stikzfig{id} \\
      \epsilon_{\ObjD{A} \TensorD R}  & \defeq \bigadapterInv \cp (\id_{\ObjD{A}} \TensorD \epsilon_R) & = & \stikzfig{definition-epsilon-recursive} \\
    \end{aligned}
  \end{equation}
\end{definition}

\begin{proposition} If $\epsilon$ is natural for $f$ and $g$, then it is natural for $f \cp g$.
  \label{proposition:epsilon-commutes-composition}
\end{proposition}

\begin{proof}
  Take morphisms $f : X \to Y$ and $g : Y \to Z$.
  By assumption, we have:
  \[
    \FunF(\FunG(f)) = \epsilon_X \cp f \cp \epsilon_Y^{-1}
    \qquad \qquad
    \FunF(\FunG(g)) = \epsilon_Y \cp g \cp \epsilon_Z^{-1}
  \]
  from which we can derive
  \begin{equation}
    \begin{aligned}
      \epsilon_X^{-1} \cp \FunF(\FunG(fg)) \cp \epsilon_Z
          = \epsilon_X^{-1} \cp \FunF(\FunG(f) \cp \FunG(g)) \cp \epsilon_Z
        & = \epsilon_X^{-1} \cp \FunF(\FunG(f)) \cp \FunF(\FunG(g)) \cp \epsilon_Z \\
        & = \epsilon_X^{-1} \cp \epsilon_X \cp f \cp \epsilon_Y^{-1} \cp \epsilon_Y \cp g \cp \epsilon_Z^{-1} \cp \epsilon_Z \\
        & = f \cp g \\
    \end{aligned}
  \end{equation}
  as required.
\end{proof}

\begin{proposition}
  \label{proposition:g-then-f-equivalence}
  \label{proposition:epsilon-natural}
  $\epsilon : \FunF \circ \FunG \rightarrow \id_{\CatD}$ is a monoidal natural isomorphism.
\end{proposition}

\begin{proof}
  We begin by showing naturality inductively, having already proven the
  inductive step for composition in Proposition
  \ref{proposition:epsilon-commutes-composition}.
  We again use Proposition \ref{proposition:sequential-normal-form}--that each
  morphism $f$ in $\CatD$ can be decomposed into `slices'
  $ f = t_1 \cp \ldots \cp t_n $
  with each $t_i$ of the form $\id_X \TensorD g_i \TensorD \id_Y$, with each
  $g_i : A \to B$ a generator.
  It thus suffices to prove that
  $t = \epsilon^{-1}_{X \TensorD A \TensorD Y} \cp \FunF(\FunG(t)) \cp \epsilon_{X \TensorD B \TensorD Y}$
  for an arbitrary `slice' $t$.
  One can check this by a second induction whose base case and inductive step correspond
  to the definition of $\FunG$ (Definition \ref{definition:fung}).
  To be precise, one can check this property graphically for each base case
  $\FunG(\id_{\ObjD{\unitC}}) \ldots \FunG(\id_{\ObjD{A}})$,
  and additionally for the inductive step $\FunG(\id_{\ObjD{A}} \TensorD q \TensorD r)$.
  Finally, note that $\epsilon$ is indeed a \emph{monoidal} natural
  transformation, which can be verified by another straightforward induction.
\end{proof}

\begin{theorem}[Mac Lane's Strictness Theorem]
  \label{theorem:strictness}
  For any monoidal category $\CatC$ there is a monoidally equivalent strict category.
\end{theorem}

\begin{proof}
  $\FunF$ and $\FunG$ are monoidal functors by Propositions
  \ref{proposition:funf-is-monoidal} and \ref{proposition:fung-is-monoidal}, and
  they form a monoidal equivalence with $\eta$ and $\epsilon$ by Propositions
  \ref{proposition:f-then-g-is-identity} and
  \ref{proposition:g-then-f-equivalence}.
  Since $\CatC$ was arbitrary, the proof is complete.
\end{proof}

Note that in contrast to Mac Lane's proof of Theorem \ref{theorem:strictness},
we make no reference to the coherence theorem.
We can therefore make use of the strictness theorem to prove coherence, which is
the subject Section \ref{section:coherence}.

\section{Symmetric Monoidal Strictness}
\label{section:symmetric-monoidal-strictness}
We now show how the strictness theorem extends to the \emph{symmetric} monoidal
case, beginning with how the braiding $\twist$ of $\CatC$ extends to $\CatD$.

\begin{definition}[Braiding of $\CatD$]
  \label{definition:catd-braiding}
  Let $\CatC$ be a symmetric monoidal category.
  The \deftext{braiding} $\twist_{X, Y}$ in $\CatD$ is defined as 
  \[
    \twist_{X, Y} \defeq \qquad \begin{tikzpicture}
	\begin{pgfonlayer}{nodelayer}
		\node [style=none] (36) at (-1, -0.5) {};
		\node [style=none] (37) at (-1, 0.5) {};
		\node [style=none] (38) at (1, 0.5) {};
		\node [style=none] (39) at (1, -0.5) {};
		\node [style=none] (40) at (-1, 0) {};
		\node [style=none] (41) at (1, 0) {};
		\node [style=none] (42) at (1.25, 0) {};
		\node [style=none] (43) at (-1.25, 0) {};
		\node [style=none] (44) at (0, 0) {$\MorD{\twist}$};
		\node [style=none] (45) at (-1.5, 0) {};
		\node [style=none] (46) at (-2, 0.5) {};
		\node [style=none] (47) at (-2, -0.5) {};
		\node [style=none] (48) at (-3, 1.5) {};
		\node [style=none] (49) at (-3, -1.5) {};
		\node [style=none] (50) at (-2.5, 0) {$\bigadapter$};
		\node [style=none] (51) at (-2, 0) {};
		\node [style=none] (52) at (-3, 0.75) {};
		\node [style=none] (53) at (-3.5, 0.75) {};
		\node [style=none] (54) at (-3, -0.75) {};
		\node [style=none] (55) at (-3.5, -0.75) {};
		\node [style=none] (56) at (1.5, 0) {};
		\node [style=none] (57) at (2, 0.5) {};
		\node [style=none] (58) at (2, -0.5) {};
		\node [style=none] (59) at (3, 1.5) {};
		\node [style=none] (60) at (3, -1.5) {};
		\node [style=none] (61) at (2.5, 0) {$\bigadapterInv$};
		\node [style=none] (62) at (2, 0) {};
		\node [style=none] (63) at (3, 0.75) {};
		\node [style=none] (64) at (3.25, 0.75) {};
		\node [style=none] (65) at (3, -0.75) {};
		\node [style=none] (66) at (3.25, -0.75) {};
		\node [style=none] (68) at (-3.75, 0.75) {};
		\node [style=none] (69) at (-6, 0.25) {};
		\node [style=none] (70) at (-6, 1.25) {};
		\node [style=none] (71) at (-4, 1.25) {};
		\node [style=none] (72) at (-4, 0.25) {};
		\node [style=none] (73) at (-6, 0.75) {};
		\node [style=none] (74) at (-4, 0.75) {};
		\node [style=none] (75) at (-3.75, 0.75) {};
		\node [style=none] (76) at (-6.75, 0.75) {};
		\node [style=none] (77) at (-5, 0.75) {$\epsilon^{-1}_X$};
		\node [style=none] (78) at (-3.75, -0.75) {};
		\node [style=none] (79) at (-6, -1.25) {};
		\node [style=none] (80) at (-6, -0.25) {};
		\node [style=none] (81) at (-4, -0.25) {};
		\node [style=none] (82) at (-4, -1.25) {};
		\node [style=none] (83) at (-6, -0.75) {};
		\node [style=none] (84) at (-4, -0.75) {};
		\node [style=none] (85) at (-3.75, -0.75) {};
		\node [style=none] (86) at (-6.75, -0.75) {};
		\node [style=none] (87) at (-5, -0.75) {$\epsilon^{-1}_Y$};
		\node [style=none] (88) at (3.5, 0.75) {};
		\node [style=none] (89) at (3.5, -0.75) {};
		\node [style=none] (90) at (3.75, 0.75) {};
		\node [style=none] (91) at (6, 0.25) {};
		\node [style=none] (92) at (6, 1.25) {};
		\node [style=none] (93) at (4, 1.25) {};
		\node [style=none] (94) at (4, 0.25) {};
		\node [style=none] (95) at (6, 0.75) {};
		\node [style=none] (96) at (4, 0.75) {};
		\node [style=none] (97) at (3.75, 0.75) {};
		\node [style=none] (98) at (6.75, 0.75) {};
		\node [style=none] (99) at (5, 0.75) {$\epsilon_Y$};
		\node [style=none] (100) at (3.75, -0.75) {};
		\node [style=none] (101) at (6, -1.25) {};
		\node [style=none] (102) at (6, -0.25) {};
		\node [style=none] (103) at (4, -0.25) {};
		\node [style=none] (104) at (4, -1.25) {};
		\node [style=none] (105) at (6, -0.75) {};
		\node [style=none] (106) at (4, -0.75) {};
		\node [style=none] (107) at (3.75, -0.75) {};
		\node [style=none] (108) at (6.75, -0.75) {};
		\node [style=none] (109) at (5, -0.75) {$\epsilon_X$};
	\end{pgfonlayer}
	\begin{pgfonlayer}{edgelayer}
		\draw (37.center) to (38.center);
		\draw (38.center) to (39.center);
		\draw (39.center) to (36.center);
		\draw (36.center) to (37.center);
		\draw (43.center) to (40.center);
		\draw (41.center) to (42.center);
		\draw (46.center) to (48.center);
		\draw (48.center) to (49.center);
		\draw (49.center) to (47.center);
		\draw (47.center) to (46.center);
		\draw (45.center) to (51.center);
		\draw (52.center) to (53.center);
		\draw (54.center) to (55.center);
		\draw (57.center) to (59.center);
		\draw (59.center) to (60.center);
		\draw (60.center) to (58.center);
		\draw (58.center) to (57.center);
		\draw (56.center) to (62.center);
		\draw (63.center) to (64.center);
		\draw (65.center) to (66.center);
		\draw (42.center) to (56.center);
		\draw (45.center) to (43.center);
		\draw (70.center) to (71.center);
		\draw (71.center) to (72.center);
		\draw (72.center) to (69.center);
		\draw (69.center) to (70.center);
		\draw (76.center) to (73.center);
		\draw (74.center) to (75.center);
		\draw (80.center) to (81.center);
		\draw (81.center) to (82.center);
		\draw (82.center) to (79.center);
		\draw (79.center) to (80.center);
		\draw (86.center) to (83.center);
		\draw (84.center) to (85.center);
		\draw (85.center) to (55.center);
		\draw (75.center) to (53.center);
		\draw (92.center) to (93.center);
		\draw (93.center) to (94.center);
		\draw (94.center) to (91.center);
		\draw (91.center) to (92.center);
		\draw (98.center) to (95.center);
		\draw (96.center) to (97.center);
		\draw (102.center) to (103.center);
		\draw (103.center) to (104.center);
		\draw (104.center) to (101.center);
		\draw (101.center) to (102.center);
		\draw (108.center) to (105.center);
		\draw (106.center) to (107.center);
		\draw (107.center) to (89.center);
		\draw (97.center) to (88.center);
		\draw (64.center) to (88.center);
		\draw (66.center) to (89.center);
	\end{pgfonlayer}
\end{tikzpicture}
  \]
  where $\epsilon$ is given in Definition \ref{definition:epsilon}.
\end{definition}


\begin{proposition}[Naturality of braiding]
  \label{proposition:braiding-naturality}
  The braiding in Definition \ref{definition:catd-braiding} is natural.
\end{proposition}
\begin{proof}
  We have that $\epsilon$ is natural by Proposition \ref{proposition:epsilon-natural}.
  The result then follows by applying naturality of $\epsilon$, adapters, and
  $\MorD{\twist}$.
\end{proof}

\begin{example}
  When $X = \ObjD{A}$ and $Y = \ObjD{B}$ naturality follows from naturality of
  adapters and the braiding in $\CatC$.
  \begin{align*}
    \begin{tikzpicture}
	\begin{pgfonlayer}{nodelayer}
		\node [style=none] (0) at (-1.5, 0) {};
		\node [style=none] (1) at (-2, 0.5) {};
		\node [style=none] (2) at (-2, -0.5) {};
		\node [style=none] (3) at (-3, 1.5) {};
		\node [style=none] (4) at (-3, -1.5) {};
		\node [style=none] (5) at (-2.5, 0) {$\bigadapter$};
		\node [style=none] (6) at (-2, 0) {};
		\node [style=none] (7) at (-3, 0.75) {};
		\node [style=none] (8) at (-4, 0.75) {};
		\node [style=none] (9) at (-3, -0.75) {};
		\node [style=none] (10) at (-4, -0.75) {};
		\node [style=none] (11) at (1.5, 0) {};
		\node [style=none] (12) at (2, 0.5) {};
		\node [style=none] (13) at (2, -0.5) {};
		\node [style=none] (14) at (3, 1.5) {};
		\node [style=none] (15) at (3, -1.5) {};
		\node [style=none] (16) at (2.5, 0) {$\bigadapterInv$};
		\node [style=none] (17) at (2, 0) {};
		\node [style=none] (18) at (3, 0.75) {};
		\node [style=none] (19) at (4, 0.75) {};
		\node [style=none] (20) at (3, -0.75) {};
		\node [style=none] (21) at (4, -0.75) {};
		\node [style=none] (22) at (-1, -0.5) {};
		\node [style=none] (23) at (-1, 0.5) {};
		\node [style=none] (24) at (1, 0.5) {};
		\node [style=none] (25) at (1, -0.5) {};
		\node [style=none] (26) at (-1, 0) {};
		\node [style=none] (27) at (1, 0) {};
		\node [style=none] (28) at (1.25, 0) {};
		\node [style=none] (29) at (-1.25, 0) {};
		\node [style=none] (30) at (0, 0) {$\MorD{\twist}$};
		\node [style=none] (42) at (-6.25, 0.25) {};
		\node [style=none] (43) at (-6.25, 1.25) {};
		\node [style=none] (44) at (-4.25, 1.25) {};
		\node [style=none] (45) at (-4.25, 0.25) {};
		\node [style=none] (46) at (-6.25, 0.75) {};
		\node [style=none] (47) at (-4.25, 0.75) {};
		\node [style=none] (48) at (-4, 0.75) {};
		\node [style=none] (49) at (-7, 0.75) {};
		\node [style=none] (50) at (-5.25, 0.75) {$\MorD{f}$};
		\node [style=none] (51) at (-6.25, -1.25) {};
		\node [style=none] (52) at (-6.25, -0.25) {};
		\node [style=none] (53) at (-4.25, -0.25) {};
		\node [style=none] (54) at (-4.25, -1.25) {};
		\node [style=none] (55) at (-6.25, -0.75) {};
		\node [style=none] (56) at (-4.25, -0.75) {};
		\node [style=none] (57) at (-4, -0.75) {};
		\node [style=none] (58) at (-7, -0.75) {};
		\node [style=none] (59) at (-5.25, -0.75) {$\MorD{g}$};
	\end{pgfonlayer}
	\begin{pgfonlayer}{edgelayer}
		\draw (1.center) to (3.center);
		\draw (3.center) to (4.center);
		\draw (4.center) to (2.center);
		\draw (2.center) to (1.center);
		\draw (0.center) to (6.center);
		\draw (7.center) to (8.center);
		\draw (9.center) to (10.center);
		\draw (12.center) to (14.center);
		\draw (14.center) to (15.center);
		\draw (15.center) to (13.center);
		\draw (13.center) to (12.center);
		\draw (11.center) to (17.center);
		\draw (18.center) to (19.center);
		\draw (20.center) to (21.center);
		\draw (23.center) to (24.center);
		\draw (24.center) to (25.center);
		\draw (25.center) to (22.center);
		\draw (22.center) to (23.center);
		\draw (29.center) to (26.center);
		\draw (27.center) to (28.center);
		\draw (28.center) to (11.center);
		\draw (0.center) to (29.center);
		\draw (43.center) to (44.center);
		\draw (44.center) to (45.center);
		\draw (45.center) to (42.center);
		\draw (42.center) to (43.center);
		\draw (49.center) to (46.center);
		\draw (47.center) to (48.center);
		\draw (52.center) to (53.center);
		\draw (53.center) to (54.center);
		\draw (54.center) to (51.center);
		\draw (51.center) to (52.center);
		\draw (58.center) to (55.center);
		\draw (56.center) to (57.center);
	\end{pgfonlayer}
\end{tikzpicture}
      & \quad = \quad \begin{tikzpicture}
	\begin{pgfonlayer}{nodelayer}
		\node [style=none] (0) at (-5, 0) {};
		\node [style=none] (1) at (-5.5, 0.5) {};
		\node [style=none] (2) at (-5.5, -0.5) {};
		\node [style=none] (3) at (-6.5, 1.5) {};
		\node [style=none] (4) at (-6.5, -1.5) {};
		\node [style=none] (5) at (-6, 0) {$\bigadapter$};
		\node [style=none] (6) at (-5.5, 0) {};
		\node [style=none] (7) at (-6.5, 0.75) {};
		\node [style=none] (8) at (-7.5, 0.75) {};
		\node [style=none] (9) at (-6.5, -0.75) {};
		\node [style=none] (10) at (-7.5, -0.75) {};
		\node [style=none] (11) at (1.5, 0) {};
		\node [style=none] (12) at (2, 0.5) {};
		\node [style=none] (13) at (2, -0.5) {};
		\node [style=none] (14) at (3, 1.5) {};
		\node [style=none] (15) at (3, -1.5) {};
		\node [style=none] (16) at (2.5, 0) {$\bigadapterInv$};
		\node [style=none] (17) at (2, 0) {};
		\node [style=none] (18) at (3, 0.75) {};
		\node [style=none] (19) at (4, 0.75) {};
		\node [style=none] (20) at (3, -0.75) {};
		\node [style=none] (21) at (4, -0.75) {};
		\node [style=none] (22) at (-1, -0.5) {};
		\node [style=none] (23) at (-1, 0.5) {};
		\node [style=none] (24) at (1, 0.5) {};
		\node [style=none] (25) at (1, -0.5) {};
		\node [style=none] (26) at (-1, 0) {};
		\node [style=none] (27) at (1, 0) {};
		\node [style=none] (28) at (1.25, 0) {};
		\node [style=none] (29) at (-1.25, 0) {};
		\node [style=none] (30) at (0, 0) {$\MorD{\twist}$};
		\node [style=none] (31) at (-4.75, 0) {};
		\node [style=none] (32) at (-1.75, 0) {};
		\node [style=none] (33) at (-4.25, -0.5) {};
		\node [style=none] (34) at (-4.25, 0.5) {};
		\node [style=none] (35) at (-2.25, 0.5) {};
		\node [style=none] (36) at (-2.25, -0.5) {};
		\node [style=none] (37) at (-4.25, 0) {};
		\node [style=none] (38) at (-2.25, 0) {};
		\node [style=none] (39) at (-2, 0) {};
		\node [style=none] (40) at (-4.5, 0) {};
		\node [style=none] (41) at (-3.25, 0) {$\MorD{f \TensorC g}$};
	\end{pgfonlayer}
	\begin{pgfonlayer}{edgelayer}
		\draw (1.center) to (3.center);
		\draw (3.center) to (4.center);
		\draw (4.center) to (2.center);
		\draw (2.center) to (1.center);
		\draw (0.center) to (6.center);
		\draw (7.center) to (8.center);
		\draw (9.center) to (10.center);
		\draw (12.center) to (14.center);
		\draw (14.center) to (15.center);
		\draw (15.center) to (13.center);
		\draw (13.center) to (12.center);
		\draw (11.center) to (17.center);
		\draw (18.center) to (19.center);
		\draw (20.center) to (21.center);
		\draw (23.center) to (24.center);
		\draw (24.center) to (25.center);
		\draw (25.center) to (22.center);
		\draw (22.center) to (23.center);
		\draw (29.center) to (26.center);
		\draw (27.center) to (28.center);
		\draw (28.center) to (11.center);
		\draw (34.center) to (35.center);
		\draw (35.center) to (36.center);
		\draw (36.center) to (33.center);
		\draw (33.center) to (34.center);
		\draw (40.center) to (37.center);
		\draw (38.center) to (39.center);
		\draw (39.center) to (32.center);
		\draw (31.center) to (40.center);
		\draw (0.center) to (31.center);
		\draw (32.center) to (29.center);
	\end{pgfonlayer}
\end{tikzpicture} \\
      & \quad = \quad \begin{tikzpicture}
	\begin{pgfonlayer}{nodelayer}
		\node [style=none] (0) at (-3.25, 0) {};
		\node [style=none] (1) at (-3.75, 0.5) {};
		\node [style=none] (2) at (-3.75, -0.5) {};
		\node [style=none] (3) at (-4.75, 1.5) {};
		\node [style=none] (4) at (-4.75, -1.5) {};
		\node [style=none] (5) at (-4.25, 0) {$\bigadapter$};
		\node [style=none] (6) at (-3.75, 0) {};
		\node [style=none] (7) at (-4.75, 0.75) {};
		\node [style=none] (8) at (-5.75, 0.75) {};
		\node [style=none] (9) at (-4.75, -0.75) {};
		\node [style=none] (10) at (-5.75, -0.75) {};
		\node [style=none] (11) at (3.25, 0) {};
		\node [style=none] (12) at (3.75, 0.5) {};
		\node [style=none] (13) at (3.75, -0.5) {};
		\node [style=none] (14) at (4.75, 1.5) {};
		\node [style=none] (15) at (4.75, -1.5) {};
		\node [style=none] (16) at (4.25, 0) {$\bigadapterInv$};
		\node [style=none] (17) at (3.75, 0) {};
		\node [style=none] (18) at (4.75, 0.75) {};
		\node [style=none] (19) at (5.75, 0.75) {};
		\node [style=none] (20) at (4.75, -0.75) {};
		\node [style=none] (21) at (5.75, -0.75) {};
		\node [style=none] (32) at (3, 0) {};
		\node [style=none] (33) at (0.5, -0.5) {};
		\node [style=none] (34) at (0.5, 0.5) {};
		\node [style=none] (35) at (2.5, 0.5) {};
		\node [style=none] (36) at (2.5, -0.5) {};
		\node [style=none] (37) at (0.5, 0) {};
		\node [style=none] (38) at (2.5, 0) {};
		\node [style=none] (39) at (2.75, 0) {};
		\node [style=none] (40) at (0.25, 0) {};
		\node [style=none] (41) at (1.5, 0) {$\MorD{g \TensorC f}$};
		\node [style=none] (42) at (-2.5, -0.5) {};
		\node [style=none] (43) at (-2.5, 0.5) {};
		\node [style=none] (44) at (-0.5, 0.5) {};
		\node [style=none] (45) at (-0.5, -0.5) {};
		\node [style=none] (46) at (-2.5, 0) {};
		\node [style=none] (47) at (-0.5, 0) {};
		\node [style=none] (48) at (-0.25, 0) {};
		\node [style=none] (49) at (-2.75, 0) {};
		\node [style=none] (50) at (-1.5, 0) {$\MorD{\twist}$};
	\end{pgfonlayer}
	\begin{pgfonlayer}{edgelayer}
		\draw (1.center) to (3.center);
		\draw (3.center) to (4.center);
		\draw (4.center) to (2.center);
		\draw (2.center) to (1.center);
		\draw (0.center) to (6.center);
		\draw (7.center) to (8.center);
		\draw (9.center) to (10.center);
		\draw (12.center) to (14.center);
		\draw (14.center) to (15.center);
		\draw (15.center) to (13.center);
		\draw (13.center) to (12.center);
		\draw (11.center) to (17.center);
		\draw (18.center) to (19.center);
		\draw (20.center) to (21.center);
		\draw (34.center) to (35.center);
		\draw (35.center) to (36.center);
		\draw (36.center) to (33.center);
		\draw (33.center) to (34.center);
		\draw (40.center) to (37.center);
		\draw (38.center) to (39.center);
		\draw (39.center) to (32.center);
		\draw (43.center) to (44.center);
		\draw (44.center) to (45.center);
		\draw (45.center) to (42.center);
		\draw (42.center) to (43.center);
		\draw (49.center) to (46.center);
		\draw (47.center) to (48.center);
		\draw (32.center) to (11.center);
		\draw (0.center) to (49.center);
		\draw (48.center) to (40.center);
	\end{pgfonlayer}
\end{tikzpicture} \\
      & \quad = \quad \begin{tikzpicture}
	\begin{pgfonlayer}{nodelayer}
		\node [style=none] (0) at (-2.75, 0) {};
		\node [style=none] (1) at (-3.25, 0.5) {};
		\node [style=none] (2) at (-3.25, -0.5) {};
		\node [style=none] (3) at (-4.25, 1.5) {};
		\node [style=none] (4) at (-4.25, -1.5) {};
		\node [style=none] (5) at (-3.75, 0) {$\bigadapter$};
		\node [style=none] (6) at (-3.25, 0) {};
		\node [style=none] (7) at (-4.25, 0.75) {};
		\node [style=none] (8) at (-5.25, 0.75) {};
		\node [style=none] (9) at (-4.25, -0.75) {};
		\node [style=none] (10) at (-5.25, -0.75) {};
		\node [style=none] (11) at (0.75, 0) {};
		\node [style=none] (12) at (1.25, 0.5) {};
		\node [style=none] (13) at (1.25, -0.5) {};
		\node [style=none] (14) at (2.25, 1.5) {};
		\node [style=none] (15) at (2.25, -1.5) {};
		\node [style=none] (16) at (1.75, 0) {$\bigadapterInv$};
		\node [style=none] (17) at (1.25, 0) {};
		\node [style=none] (18) at (2.25, 0.75) {};
		\node [style=none] (19) at (3.25, 0.75) {};
		\node [style=none] (20) at (2.25, -0.75) {};
		\node [style=none] (21) at (3.25, -0.75) {};
		\node [style=none] (42) at (-2, -0.5) {};
		\node [style=none] (43) at (-2, 0.5) {};
		\node [style=none] (44) at (0, 0.5) {};
		\node [style=none] (45) at (0, -0.5) {};
		\node [style=none] (46) at (-2, 0) {};
		\node [style=none] (47) at (0, 0) {};
		\node [style=none] (48) at (0.25, 0) {};
		\node [style=none] (49) at (-2.25, 0) {};
		\node [style=none] (50) at (-1, 0) {$\MorD{\twist}$};
		\node [style=none] (51) at (3.5, 0.25) {};
		\node [style=none] (52) at (3.5, 1.25) {};
		\node [style=none] (53) at (5.5, 1.25) {};
		\node [style=none] (54) at (5.5, 0.25) {};
		\node [style=none] (55) at (3.5, 0.75) {};
		\node [style=none] (56) at (5.5, 0.75) {};
		\node [style=none] (57) at (6, 0.75) {};
		\node [style=none] (58) at (3.25, 0.75) {};
		\node [style=none] (59) at (4.5, 0.75) {$\MorD{g}$};
		\node [style=none] (60) at (3.5, -1.25) {};
		\node [style=none] (61) at (3.5, -0.25) {};
		\node [style=none] (62) at (5.5, -0.25) {};
		\node [style=none] (63) at (5.5, -1.25) {};
		\node [style=none] (64) at (3.5, -0.75) {};
		\node [style=none] (65) at (5.5, -0.75) {};
		\node [style=none] (66) at (6, -0.75) {};
		\node [style=none] (67) at (3.25, -0.75) {};
		\node [style=none] (68) at (4.5, -0.75) {$\MorD{f}$};
	\end{pgfonlayer}
	\begin{pgfonlayer}{edgelayer}
		\draw (1.center) to (3.center);
		\draw (3.center) to (4.center);
		\draw (4.center) to (2.center);
		\draw (2.center) to (1.center);
		\draw (0.center) to (6.center);
		\draw (7.center) to (8.center);
		\draw (9.center) to (10.center);
		\draw (12.center) to (14.center);
		\draw (14.center) to (15.center);
		\draw (15.center) to (13.center);
		\draw (13.center) to (12.center);
		\draw (11.center) to (17.center);
		\draw (18.center) to (19.center);
		\draw (20.center) to (21.center);
		\draw (43.center) to (44.center);
		\draw (44.center) to (45.center);
		\draw (45.center) to (42.center);
		\draw (42.center) to (43.center);
		\draw (49.center) to (46.center);
		\draw (47.center) to (48.center);
		\draw (0.center) to (49.center);
		\draw (48.center) to (11.center);
		\draw (52.center) to (53.center);
		\draw (53.center) to (54.center);
		\draw (54.center) to (51.center);
		\draw (51.center) to (52.center);
		\draw (58.center) to (55.center);
		\draw (56.center) to (57.center);
		\draw (61.center) to (62.center);
		\draw (62.center) to (63.center);
		\draw (63.center) to (60.center);
		\draw (60.center) to (61.center);
		\draw (67.center) to (64.center);
		\draw (65.center) to (66.center);
	\end{pgfonlayer}
\end{tikzpicture} \\
  \end{align*}
\end{example}

This braiding makes $\CatD$ symmetric monoidal.

\begin{proposition}
  \label{proposition:catd-symmetric-monoidal}
  If $\CatC$ is symmetric monoidal then $\CatD$ is symmetric monoidal.
\end{proposition}
\begin{proof}
  Let the braiding of $\CatD$ be defined as in \ref{definition:catd-braiding}.
  It is natural by Proposition \ref{proposition:braiding-naturality},
  and $\twist_{X, Y} \cp \twist_{Y, X} = \id_{X \TensorD Y}$
  because adapters, $\epsilon$, and $\MorD{\twist}$ are all isomorphisms.
  Finally, we must show that the unitor coherence and associator coherence axioms
  of symmetric monoidal categories are satisfied.
  The unit coherence follows straightforwardly.
  Calculating for $\twist_{X, \unitD}$ we have
  \begin{align*}
    \twist_{X, \unitD} & = \qquad \begin{tikzpicture}
	\begin{pgfonlayer}{nodelayer}
		\node [style=none] (0) at (-1, -0.5) {};
		\node [style=none] (1) at (-1, 0.5) {};
		\node [style=none] (2) at (1, 0.5) {};
		\node [style=none] (3) at (1, -0.5) {};
		\node [style=none] (4) at (-1, 0) {};
		\node [style=none] (5) at (1, 0) {};
		\node [style=none] (6) at (1.25, 0) {};
		\node [style=none] (7) at (-1.25, 0) {};
		\node [style=none] (8) at (0, 0) {$\MorD{\twist}$};
		\node [style=none] (9) at (-1.5, 0) {};
		\node [style=none] (10) at (-2, 0.5) {};
		\node [style=none] (11) at (-2, -0.5) {};
		\node [style=none] (12) at (-3, 1.5) {};
		\node [style=none] (13) at (-3, -1.5) {};
		\node [style=none] (14) at (-2.5, 0) {$\bigadapter$};
		\node [style=none] (15) at (-2, 0) {};
		\node [style=none] (16) at (-3, 0.75) {};
		\node [style=none] (17) at (-3.5, 0.75) {};
		\node [style=none] (18) at (-3, -0.75) {};
		\node [style=none] (19) at (-3.5, -0.75) {};
		\node [style=none] (20) at (1.5, 0) {};
		\node [style=none] (21) at (2, 0.5) {};
		\node [style=none] (22) at (2, -0.5) {};
		\node [style=none] (23) at (3, 1.5) {};
		\node [style=none] (24) at (3, -1.5) {};
		\node [style=none] (25) at (2.5, 0) {$\bigadapterInv$};
		\node [style=none] (26) at (2, 0) {};
		\node [style=none] (27) at (3, 0.75) {};
		\node [style=none] (28) at (3.5, 0.75) {};
		\node [style=none] (29) at (3, -0.75) {};
		\node [style=none] (30) at (3.5, -0.75) {};
		\node [style=none] (31) at (4.5, 1.75) {};
		\node [style=none] (32) at (4.5, -0.25) {};
		\node [style=none] (33) at (5, 0.75) {$\adapterInv$};
		\node [style=none] (34) at (4.5, 0.75) {};
		\node [style=none] (35) at (3.5, 0.75) {};
		\node [style=none] (36) at (5.75, 0.75) {};
		\node [style=none] (37) at (-4.5, 0.25) {};
		\node [style=none] (38) at (-4.5, -1.75) {};
		\node [style=none] (39) at (-5, -0.75) {$\adapter$};
		\node [style=none] (40) at (-4.5, -0.75) {};
		\node [style=none] (41) at (-3.5, -0.75) {};
		\node [style=none] (42) at (-5.75, -0.75) {};
		\node [style=none] (43) at (-6.5, 0.75) {};
		\node [style=none] (44) at (6.5, -0.75) {};
	\end{pgfonlayer}
	\begin{pgfonlayer}{edgelayer}
		\draw (1.center) to (2.center);
		\draw (2.center) to (3.center);
		\draw (3.center) to (0.center);
		\draw (0.center) to (1.center);
		\draw (7.center) to (4.center);
		\draw (5.center) to (6.center);
		\draw (10.center) to (12.center);
		\draw (12.center) to (13.center);
		\draw (13.center) to (11.center);
		\draw (11.center) to (10.center);
		\draw (9.center) to (15.center);
		\draw (16.center) to (17.center);
		\draw (18.center) to (19.center);
		\draw (21.center) to (23.center);
		\draw (23.center) to (24.center);
		\draw (24.center) to (22.center);
		\draw (22.center) to (21.center);
		\draw (20.center) to (26.center);
		\draw (27.center) to (28.center);
		\draw (29.center) to (30.center);
		\draw (6.center) to (20.center);
		\draw (9.center) to (7.center);
		\draw (31.center) to (32.center);
		\draw (34.center) to (35.center);
		\draw (36.center) to (31.center);
		\draw (32.center) to (36.center);
		\draw (37.center) to (38.center);
		\draw (40.center) to (41.center);
		\draw (42.center) to (37.center);
		\draw (38.center) to (42.center);
		\draw (43.center) to (17.center);
		\draw (30.center) to (44.center);
	\end{pgfonlayer}
\end{tikzpicture} \\ \\
                       & = \qquad \begin{tikzpicture}
	\begin{pgfonlayer}{nodelayer}
		\node [style=none] (0) at (-1, 0.5) {};
		\node [style=none] (1) at (-1, -0.5) {};
		\node [style=none] (2) at (1, -0.5) {};
		\node [style=none] (3) at (1, 0.5) {};
		\node [style=none] (4) at (-1, 0) {};
		\node [style=none] (5) at (1, 0) {};
		\node [style=none] (6) at (1.25, 0) {};
		\node [style=none] (7) at (-1.25, 0) {};
		\node [style=none] (8) at (0, 0) {$\MorD{\twist}$};
		\node [style=none] (9) at (-1.5, 0) {};
		\node [style=none] (20) at (1.5, 0) {};
		\node [style=none] (45) at (-4.25, 0.5) {};
		\node [style=none] (46) at (-4.25, -0.5) {};
		\node [style=none] (47) at (-2.25, -0.5) {};
		\node [style=none] (48) at (-2.25, 0.5) {};
		\node [style=none] (49) at (-4.25, 0) {};
		\node [style=none] (50) at (-2.25, 0) {};
		\node [style=none] (51) at (-2, 0) {};
		\node [style=none] (52) at (-4.5, 0) {};
		\node [style=none] (53) at (-3.25, 0) {$\MorD{\unitl^{-1}}$};
		\node [style=none] (54) at (-4.75, 0) {};
		\node [style=none] (55) at (2.25, 0.5) {};
		\node [style=none] (56) at (2.25, -0.5) {};
		\node [style=none] (57) at (4.25, -0.5) {};
		\node [style=none] (58) at (4.25, 0.5) {};
		\node [style=none] (59) at (2.25, 0) {};
		\node [style=none] (60) at (4.25, 0) {};
		\node [style=none] (61) at (4.5, 0) {};
		\node [style=none] (62) at (2, 0) {};
		\node [style=none] (63) at (3.25, 0) {$\MorD{\unitr}$};
		\node [style=none] (64) at (1.75, 0) {};
		\node [style=none] (65) at (4.75, 0) {};
	\end{pgfonlayer}
	\begin{pgfonlayer}{edgelayer}
		\draw (1.center) to (2.center);
		\draw (2.center) to (3.center);
		\draw (3.center) to (0.center);
		\draw (0.center) to (1.center);
		\draw (7.center) to (4.center);
		\draw (5.center) to (6.center);
		\draw (6.center) to (20.center);
		\draw (9.center) to (7.center);
		\draw (46.center) to (47.center);
		\draw (47.center) to (48.center);
		\draw (48.center) to (45.center);
		\draw (45.center) to (46.center);
		\draw (52.center) to (49.center);
		\draw (50.center) to (51.center);
		\draw (54.center) to (52.center);
		\draw (51.center) to (9.center);
		\draw (56.center) to (57.center);
		\draw (57.center) to (58.center);
		\draw (58.center) to (55.center);
		\draw (55.center) to (56.center);
		\draw (62.center) to (59.center);
		\draw (60.center) to (61.center);
		\draw (61.center) to (65.center);
		\draw (64.center) to (62.center);
		\draw (20.center) to (64.center);
	\end{pgfonlayer}
\end{tikzpicture} \\ \\
                       & = \qquad \begin{tikzpicture}
	\begin{pgfonlayer}{nodelayer}
		\node [style=none] (66) at (-2.5, 0) {};
		\node [style=none] (67) at (2.5, 0) {};
	\end{pgfonlayer}
	\begin{pgfonlayer}{edgelayer}
		\draw (66.center) to (67.center);
	\end{pgfonlayer}
\end{tikzpicture}

  \end{align*}
  where the final step follows from the unitor coherence in $\CatC$.
  The case of $\twist_{X, \unitD}$ holds in essentially the same way.

  By similar calculations one may show that the associator coherence holds.
  Essentially, the proof follows by naturality and the associator coherence of
  $\CatC$.
\end{proof}

In addition, the monoidal functor $\FunS$ extends to a \emph{symmetric} monoidal functor.

\begin{proposition}
  \label{proposition:funf-symmetric-monoidal}
  If $\CatC$ is symmetric monoidal, then $\FunS : \CatC \to \CatD$ is a
  symmetric monoidal functor.
\end{proposition}
\begin{proof}
  For $\FunS$ to be symmetric monoidal, we require that
  $\twist_{\FunF(A), \FunF(B)} = \bigadapter \cp \FunS(\twist_{A, B}) \cp \bigadapterInv$.
  This is immediate if we simply apply $\FunS$, yielding the equality
  \[ \twist_{\FunF(A), \FunF(B)} \qquad = \qquad \twist_{\ObjD{A}, \ObjD{B}} \qquad = \qquad \begin{tikzpicture}
	\begin{pgfonlayer}{nodelayer}
		\node [style=none] (0) at (-1.5, 0) {};
		\node [style=none] (1) at (-2, 0.5) {};
		\node [style=none] (2) at (-2, -0.5) {};
		\node [style=none] (3) at (-3, 1.5) {};
		\node [style=none] (4) at (-3, -1.5) {};
		\node [style=none] (5) at (-2.5, 0) {$\bigadapter$};
		\node [style=none] (6) at (-2, 0) {};
		\node [style=none] (7) at (-3, 0.75) {};
		\node [style=none] (9) at (-3, -0.75) {};
		\node [style=none] (11) at (1.5, 0) {};
		\node [style=none] (12) at (2, 0.5) {};
		\node [style=none] (13) at (2, -0.5) {};
		\node [style=none] (14) at (3, 1.5) {};
		\node [style=none] (15) at (3, -1.5) {};
		\node [style=none] (16) at (2.5, 0) {$\bigadapterInv$};
		\node [style=none] (17) at (2, 0) {};
		\node [style=none] (18) at (3, 0.75) {};
		\node [style=none] (19) at (4, 0.75) {};
		\node [style=none] (20) at (3, -0.75) {};
		\node [style=none] (21) at (4, -0.75) {};
		\node [style=none] (22) at (-1, -0.5) {};
		\node [style=none] (23) at (-1, 0.5) {};
		\node [style=none] (24) at (1, 0.5) {};
		\node [style=none] (25) at (1, -0.5) {};
		\node [style=none] (26) at (-1, 0) {};
		\node [style=none] (27) at (1, 0) {};
		\node [style=none] (28) at (1.25, 0) {};
		\node [style=none] (29) at (-1.25, 0) {};
		\node [style=none] (30) at (0, 0) {$\MorD{\twist}$};
		\node [style=none] (31) at (-4, 0.75) {};
		\node [style=none] (32) at (-3, 0.75) {};
		\node [style=none] (33) at (-4, -0.75) {};
		\node [style=none] (34) at (-3, -0.75) {};
	\end{pgfonlayer}
	\begin{pgfonlayer}{edgelayer}
		\draw (1.center) to (3.center);
		\draw (3.center) to (4.center);
		\draw (4.center) to (2.center);
		\draw (2.center) to (1.center);
		\draw (0.center) to (6.center);
		\draw (12.center) to (14.center);
		\draw (14.center) to (15.center);
		\draw (15.center) to (13.center);
		\draw (13.center) to (12.center);
		\draw (11.center) to (17.center);
		\draw (18.center) to (19.center);
		\draw (20.center) to (21.center);
		\draw (23.center) to (24.center);
		\draw (24.center) to (25.center);
		\draw (25.center) to (22.center);
		\draw (22.center) to (23.center);
		\draw (29.center) to (26.center);
		\draw (27.center) to (28.center);
		\draw (28.center) to (11.center);
		\draw (0.center) to (29.center);
		\draw (31.center) to (32.center);
		\draw (33.center) to (34.center);
	\end{pgfonlayer}
\end{tikzpicture} \]
  which holds by Definition \ref{definition:catd-braiding}.
\end{proof}

Finally, for $\CatC$ and $\CatD$ to be symmetric monoidally equivalent,
we must also have that $\FunG$ is a symmetric monoidal functor.

\begin{proposition}
  \label{proposition:fung-symmetric-monoidal}
  If $\CatC$ is symmetric monoidal, then $\FunN : \CatD \to \CatC$ is a
  symmetric monoidal functor.
\end{proposition}
\begin{proof}
  One can check by straightforward induction that
  $\FunN(\epsilon_X \TensorD \epsilon_Y) = \Psi_{X, Y}$.
  Having done so, the result is immediate:
  \[
    \FunN\left(\begin{tikzpicture}
	\begin{pgfonlayer}{nodelayer}
		\node [style=none] (36) at (-1, -0.5) {};
		\node [style=none] (37) at (-1, 0.5) {};
		\node [style=none] (38) at (1, 0.5) {};
		\node [style=none] (39) at (1, -0.5) {};
		\node [style=none] (40) at (-1, 0) {};
		\node [style=none] (41) at (1, 0) {};
		\node [style=none] (42) at (1.25, 0) {};
		\node [style=none] (43) at (-1.25, 0) {};
		\node [style=none] (44) at (0, 0) {$\MorD{\twist}$};
		\node [style=none] (45) at (-1.5, 0) {};
		\node [style=none] (46) at (-2, 0.5) {};
		\node [style=none] (47) at (-2, -0.5) {};
		\node [style=none] (48) at (-3, 1.5) {};
		\node [style=none] (49) at (-3, -1.5) {};
		\node [style=none] (50) at (-2.5, 0) {$\bigadapter$};
		\node [style=none] (51) at (-2, 0) {};
		\node [style=none] (52) at (-3, 0.75) {};
		\node [style=none] (53) at (-3.5, 0.75) {};
		\node [style=none] (54) at (-3, -0.75) {};
		\node [style=none] (55) at (-3.5, -0.75) {};
		\node [style=none] (56) at (1.5, 0) {};
		\node [style=none] (57) at (2, 0.5) {};
		\node [style=none] (58) at (2, -0.5) {};
		\node [style=none] (59) at (3, 1.5) {};
		\node [style=none] (60) at (3, -1.5) {};
		\node [style=none] (61) at (2.5, 0) {$\bigadapterInv$};
		\node [style=none] (62) at (2, 0) {};
		\node [style=none] (63) at (3, 0.75) {};
		\node [style=none] (64) at (3.25, 0.75) {};
		\node [style=none] (65) at (3, -0.75) {};
		\node [style=none] (66) at (3.25, -0.75) {};
		\node [style=none] (68) at (-3.75, 0.75) {};
		\node [style=none] (69) at (-6, 0.25) {};
		\node [style=none] (70) at (-6, 1.25) {};
		\node [style=none] (71) at (-4, 1.25) {};
		\node [style=none] (72) at (-4, 0.25) {};
		\node [style=none] (73) at (-6, 0.75) {};
		\node [style=none] (74) at (-4, 0.75) {};
		\node [style=none] (75) at (-3.75, 0.75) {};
		\node [style=none] (76) at (-6.75, 0.75) {};
		\node [style=none] (77) at (-5, 0.75) {$\epsilon^{-1}_X$};
		\node [style=none] (78) at (-3.75, -0.75) {};
		\node [style=none] (79) at (-6, -1.25) {};
		\node [style=none] (80) at (-6, -0.25) {};
		\node [style=none] (81) at (-4, -0.25) {};
		\node [style=none] (82) at (-4, -1.25) {};
		\node [style=none] (83) at (-6, -0.75) {};
		\node [style=none] (84) at (-4, -0.75) {};
		\node [style=none] (85) at (-3.75, -0.75) {};
		\node [style=none] (86) at (-6.75, -0.75) {};
		\node [style=none] (87) at (-5, -0.75) {$\epsilon^{-1}_Y$};
		\node [style=none] (88) at (3.5, 0.75) {};
		\node [style=none] (89) at (3.5, -0.75) {};
		\node [style=none] (90) at (3.75, 0.75) {};
		\node [style=none] (91) at (6, 0.25) {};
		\node [style=none] (92) at (6, 1.25) {};
		\node [style=none] (93) at (4, 1.25) {};
		\node [style=none] (94) at (4, 0.25) {};
		\node [style=none] (95) at (6, 0.75) {};
		\node [style=none] (96) at (4, 0.75) {};
		\node [style=none] (97) at (3.75, 0.75) {};
		\node [style=none] (98) at (6.75, 0.75) {};
		\node [style=none] (99) at (5, 0.75) {$\epsilon_Y$};
		\node [style=none] (100) at (3.75, -0.75) {};
		\node [style=none] (101) at (6, -1.25) {};
		\node [style=none] (102) at (6, -0.25) {};
		\node [style=none] (103) at (4, -0.25) {};
		\node [style=none] (104) at (4, -1.25) {};
		\node [style=none] (105) at (6, -0.75) {};
		\node [style=none] (106) at (4, -0.75) {};
		\node [style=none] (107) at (3.75, -0.75) {};
		\node [style=none] (108) at (6.75, -0.75) {};
		\node [style=none] (109) at (5, -0.75) {$\epsilon_X$};
	\end{pgfonlayer}
	\begin{pgfonlayer}{edgelayer}
		\draw (37.center) to (38.center);
		\draw (38.center) to (39.center);
		\draw (39.center) to (36.center);
		\draw (36.center) to (37.center);
		\draw (43.center) to (40.center);
		\draw (41.center) to (42.center);
		\draw (46.center) to (48.center);
		\draw (48.center) to (49.center);
		\draw (49.center) to (47.center);
		\draw (47.center) to (46.center);
		\draw (45.center) to (51.center);
		\draw (52.center) to (53.center);
		\draw (54.center) to (55.center);
		\draw (57.center) to (59.center);
		\draw (59.center) to (60.center);
		\draw (60.center) to (58.center);
		\draw (58.center) to (57.center);
		\draw (56.center) to (62.center);
		\draw (63.center) to (64.center);
		\draw (65.center) to (66.center);
		\draw (42.center) to (56.center);
		\draw (45.center) to (43.center);
		\draw (70.center) to (71.center);
		\draw (71.center) to (72.center);
		\draw (72.center) to (69.center);
		\draw (69.center) to (70.center);
		\draw (76.center) to (73.center);
		\draw (74.center) to (75.center);
		\draw (80.center) to (81.center);
		\draw (81.center) to (82.center);
		\draw (82.center) to (79.center);
		\draw (79.center) to (80.center);
		\draw (86.center) to (83.center);
		\draw (84.center) to (85.center);
		\draw (85.center) to (55.center);
		\draw (75.center) to (53.center);
		\draw (92.center) to (93.center);
		\draw (93.center) to (94.center);
		\draw (94.center) to (91.center);
		\draw (91.center) to (92.center);
		\draw (98.center) to (95.center);
		\draw (96.center) to (97.center);
		\draw (102.center) to (103.center);
		\draw (103.center) to (104.center);
		\draw (104.center) to (101.center);
		\draw (101.center) to (102.center);
		\draw (108.center) to (105.center);
		\draw (106.center) to (107.center);
		\draw (107.center) to (89.center);
		\draw (97.center) to (88.center);
		\draw (64.center) to (88.center);
		\draw (66.center) to (89.center);
	\end{pgfonlayer}
\end{tikzpicture}\right) = \Psi^{-1}_{X, Y} \cp \twist \cp \Psi_{Y, X}
  \]
  as required.
\end{proof}

\section{Coherence}
\label{section:coherence}
\topicsentence{
  We can now give an elementary proof of Mac Lane's \emph{coherence theorem}.
}
In~\cite{CFWM}, Mac Lane gives his theorem in two parts:
Theorem 1~\cite[p.~166]{CFWM} and its corollary~\cite[p.~169]{CFWM}.
The `meat' of the proof is in the former part, corresponding to our Section
\ref{section:preorder-proof}.
A graphical exposition of Mac Lane's proof of the corollary is given in
Section \ref{section:corollary-proof}.

\topicsentence{
  Mac Lane begins by defining a certain preorder $\cat{W}$, which he then shows
  enjoys the following property:
}

\begin{theorem} (Mac Lane's Coherence Theorem~\cite[p.~166]{CFWM}) \\
  \label{theorem:coherence}
  Let $\cat{M}$ be an arbitrary monoidal category, and let $M$ be an object of $\cat{M}$.
  Then there is a unique strict monoidal functor $\CatW \to \cat{M}$ such that
  $W \mapsto M$.
\end{theorem}

In contrast, we will define $\CatW$ so this unique functor is easy to construct,
and then use $\CatSW$ to give a \emph{graphical proof} that $\CatW$ is a
preorder.
Note that the monoidal functor in question is \emph{strict}, so its coherence
maps are identities.

\subsection{The free monoidal category on one generator}
\label{section:preorder-proof}
We begin by defining $\CatW$. Again, recall that our definition differs from Mac
Lane; we will later show that this definition indeed yields a preorder in order
to guarantee that we indeed prove the same theorem.

\begin{definition} 
  \label{definition:catw}
  We define $\CatW$ as the monoidal category freely generated by a single object
  $W$ and no morphisms except those required by the definition of a monoidal
  category.
  \footnote{Mac Lane denotes the generating object as $(-)$ to suggest an ``empty place''.
  We follow Peter Hines' convention~\cite{hines2015coherence} and use $W$ instead.}
\end{definition}

\begin{remark}
  The objects of $\CatW$ are $\unit_{\CatW}$, $W$, and their tensor products.
  The arrows are $\id, \rho, \lambda, \alpha$ and their composites and tensor products.
\end{remark}

\topicsentence{It is now clear that the statement of Mac Lane's Theorem 1 holds for our definition of \nolinebreak $\CatW$}:

\begin{proposition}
  \label{proposition:catw-unique-strict-monoidal-functor}
  Given an arbitrary monoidal category $\cat{M}$ and object $M \in \cat{M}$,
  there is a unique strict monoidal functor $\CatW \to \cat{M}$ with $W \mapsto M$.
\end{proposition}

\begin{proof}
  Suppose $\FunU : \CatW \to \cat{M}$ is such a (strict) monoidal functor.
  Then we must have that $\FunU(W) = M$ by assumption, and
  \[
    \FunU(\unit) = \unit
    \quad
    \FunU(A \otimes M) = \FunU(A) \otimes \FunU(M)
    \quad
    \FunU(f) = f, f \in \{\assoc, \unitl, \unitr, \id \}
    \quad
    \FunU(f \otimes g) = \FunU(f) \otimes \FunU(g)
  \]
  because $\FunU$ is strict.
  But this accounts for all objects and morphisms of $\CatW$, and so $\FunU$ must be unique.
\end{proof}


\topicsentence{
  However, to constitute a proof of the coherence theorem we must now \emph{prove} that
  $\CatW$ is a preorder.
}
Our argument proceeds in three main steps. We will show the following:

\begin{enumerate}
  \item For any monoidal category $\CatC$, If $\CatD$ is a preorder, then so is $\CatC$
  \item $\CatSW$ is generated solely by adapters $\{ \bigadapter, \adapter \}$ and their inverses.
  \item $\CatSW$ is a preorder (which we prove graphically)
\end{enumerate}

The first two steps are straightforward; we address them now. The third requires
more work, and is contained in Section \ref{section:catsw-is-a-preorder}.

\begin{proposition} If \CatD is a preorder, then so is $\CatC$.
  \label{proposition:if-catd-preorder-then-catc-preorder}
\end{proposition}

\begin{proof}
  Let $f, g : \CatC(A, B)$.
  Recall that $\FunG \circ \FunF = \id$, and so we can derive
  $ f = \FunG(\FunF(f)) = \FunG(\FunF(g)) = g $
  where we used that $\FunF(f) = \FunF(g)$ because $\CatD$ is a preorder.
\end{proof}

%
%

Another lemma shows we can reason about $\CatSW$ by considering only adapters:

\begin{proposition} $\CatSW$ is generated by $\bigadapter, \adapter$ and their inverses.
  \label{proposition:catsw-generated-by-adapters}
\end{proposition}

\begin{proof}
  Arrows of $\CatSW$ are by definition either adapters $\bigadapter, \adapter$, their inverses,
  or morphisms $\MorD{f}$ for some $f \in \CatW$.
  But note that all such $f \in \CatW$ are either $\id, \rho, \lambda, \alpha$ or their composites.
  It is clear that each of $\MorD{\unitl}, \MorD{\unitr}, \MorD{\assoc}$ can each be written as
  adapters by equations \eqref{equation:catd-monoidal-equations}, so it remains
  to show that composites of such morphisms can also be written this way.

  That is, we must show that $\FunF(f \otimes g)$ can be expressed using only
  adapters and their composites.
  This can be proved inductively: if $\FunF(f), \FunF(g)$ can be expressed using adapters,
  then so too can compositions $\FunF(f \cp g) = \FunF(f) \cp \FunF(g)$
  and tensors $\FunF(f \otimes g) = \bigadapter \cp (\FunF(f) \TensorD \FunF(g)) \cp \bigadapterInv$.

  Thus every morphism of $\CatSW$ can be expressed in terms of adapters,
  and so the category can be said to be \emph{generated} by (only) adapters.
\end{proof}

\subsection{Graphical proof that \texorpdfstring{$\CatSW$}{W} is a preorder}
\label{section:catsw-is-a-preorder}
\topicsentence{
  We can now prove graphically that $\CatSW$ is a preorder using a normal form argument.
} Our approach is as follows:

\begin{enumerate}
  \item Define for each object a $\objsize$ in $\Nat$ (Definition \ref{definition:objsize})
  \item Prove all morphisms in $\CatSW$ go between objects of the same size (Proposition \ref{proposition:catsw-morphisms-preserve-size})
  \item Define a canonical arrow $\canonical(A, B)$ between any two objects of the same size (Definition \ref{definition:canonical-arrows})
  \item Show that any arrow is equal to the canonical one (Proposition \ref{proposition:all-morphisms-canonical})
\end{enumerate}

Note that we make heavy use of Proposition
\ref{proposition:catsw-generated-by-adapters}, which lets us reason about
$\CatSW$ inductively in terms of adapters and their tensors and composites.

We begin--following Mac Lane--by defining the \emph{size} of an object
(the same as Mac Lane's notion of \emph{length}~\cite[p.~165]{CFWM})
as follows:

\begin{definition} 
  \label{definition:objsize}
  We define the $\objsize$ of an object as the number of occurrences of $W$, defined inductively:
  \begin{gather*}
    \objsize(\unitSW)             \defeq 0
    \qquad
    \objsize(\ObjD{\unitW})       \defeq 0
    \qquad
    \objsize(\ObjD{W})            \defeq 1
    \\
    \objsize(\ObjD{A \TensorC B}) \defeq \objsize(A) + \objsize(B)
    \qquad
    \objsize(X \TensorD Y)        \defeq \objsize(X) + \objsize(Y)
  \end{gather*}
\end{definition}

\begin{proposition} $\CatSW$ morphisms preserve $\objsize$: \label{proposition:catsw-morphisms-preserve-size} \\
  If $f : A \to B$ is a morphism in $\CatSW$, then $\objsize(A) = \objsize(B)$.
\end{proposition}

\begin{proof} Induction on morphisms.
\end{proof}

\topicsentence{
  We will define the canonical arrow $\canonical(A, B)$ in two halves, $\objpack$ and $\objunpack$.
}
To do so, we will first need some additional definitions.

\begin{definition} 
  \label{definition:pack}
  We define the `packing' and `unpacking' morphisms $\objpack$ and $\objunpack$ in terms of objects of $\CatSW$.
  Let $A \in \CatSW$ be an object.
  Then $\objpack(A)$ is the morphism defined inductively as follows:

  \begin{gather*}
    \objpack(\unitSW)              \quad \defeq \quad \stikzfig{empty}
    \qquad
    \objpack(\ObjD{\unitW})        \quad \defeq \quad \stikzfig{pack-adapter}
    \qquad
    \objpack(\ObjD{W})             \quad \defeq \quad \stikzfig{pack-id} 
    \\
    \objpack(\ObjD{A \TensorC B})  \quad \defeq \quad \stikzfig{pack-tensorc}
    \qquad
    \objpack(X \TensorD Y)         \quad \defeq \quad \stikzfig{pack-tensord}
  \end{gather*}

  And define $\objunpack(A)$ as $\objpack(A)^{-1}$.
\end{definition}

\begin{remark}
  It can be more intuitive to define $\objunpack$ first, thinking of it as the
  adapter which removes extraneous $\unitC$ objects and `normalises' the object
  into a flat array of $\ObjD{W}$ objects.
  In this view, $\objpack$ is the adapter morphism taking a fixed number of
  $\ObjD{W}$ objects and assembling them into a certain bracketing, with unit
  objects inserted as appropriate.
\end{remark}

In Definition \ref{definition:pack} we implicitly used that $\CatSW$ is a
groupoid to define $\objunpack$, which we now prove:

\begin{proposition} $\CatSW$ is a groupoid.
\end{proposition}
\begin{proof}
  Generators and identities have inverses by Definition
  \ref{definition:catd},
  which allows an inductive definition for tensor and composition, i.e.
  $ (f \cp g)^{-1} = g^{-1} f^{-1} $
  and
  $ (f \TensorD g)^{-1} = f^{-1} \TensorD g^{-1} $
  respectively.
\end{proof}



\topicsentence{
  Now, in order to define the canonical arrow as a composition of $\objpack$ and
  $\objunpack$, we will need the following lemma which states that for objects
  of the same size, we can compose their $\objunpack$ and $\objpack$ morphisms.
}

\begin{proposition} $\objpack(A) : \ObjD{W}^{\objsize(A)} \to A$.\\
  In other words, for an object $A$ of size $n$, the domain of $\objpack(A)$ is
  the $n$-fold $\TensorD$-tensoring of $\ObjD{W}$.
\end{proposition}

\begin{proof}
  Simple induction on objects (the domain of each $\objpack(A)$ is either $\unitSW$, $\ObjD{W}^k$ or a tensoring of terms)
\end{proof}

%

\begin{definition} 
  \label{definition:canonical-arrows}
  To each pair of objects $A$, $B$ of the same size, we can define a canonical arrow as follows:
  \[ \canonical(A, B) \defeq \objunpack(A) \cp \objpack(B) \]
\end{definition}

Note that the composition of Definition \ref{definition:canonical-arrows} is well-typed because
$\objsize(A) = \objsize(B)$ by Proposition \ref{proposition:catsw-morphisms-preserve-size}:
$
  \cod(\objunpack(A)) = \ObjD{W}^{\objsize(A)}
                      = \ObjD{W}^{\objsize(B)}
                      = \dom(\objpack(B))
$

\begin{example}
  The canonical arrow between $W \TensorW (\unitW \TensorW W)$ and $(W \TensorW
  \unitW) \TensorW W$ is
  $ \scalebox{0.5}{\begin{tikzpicture}
	\begin{pgfonlayer}{nodelayer}
		\node [style=none] (0) at (3.25, -0.5) {};
		\node [style=none] (1) at (3.25, 0.5) {};
		\node [style=none] (2) at (2.75, -1) {};
		\node [style=none] (3) at (2.75, 1) {};
		\node [style=none] (4) at (3.25, 0) {};
		\node [style=none] (5) at (2.75, -0.5) {};
		\node [style=none] (6) at (2.75, 0.5) {};
		\node [style=none] (7) at (1.75, -1) {};
		\node [style=none] (8) at (2.5, 0.5) {};
		\node [style=none] (9) at (1.75, 0) {};
		\node [style=none] (10) at (1.75, 1) {};
		\node [style=none] (11) at (1.25, -0.5) {};
		\node [style=none] (12) at (1.25, 1.5) {};
		\node [style=none] (13) at (1.75, 0.5) {};
		\node [style=none] (14) at (1.25, 0) {};
		\node [style=none] (15) at (1.25, 1) {};
		\node [style=none] (16) at (2, 0.5) {};
		\node [style=none] (18) at (0.25, 1) {};
		\node [style=none] (19) at (0.25, -1) {};
		\node [style=none] (20) at (3.75, 0) {};
		\node [style=none] (21) at (-3.25, 0.5) {};
		\node [style=none] (22) at (-3.25, -0.5) {};
		\node [style=none] (23) at (-2.75, 1) {};
		\node [style=none] (24) at (-2.75, -1) {};
		\node [style=none] (25) at (-3.25, 0) {};
		\node [style=none] (26) at (-2.75, 0.5) {};
		\node [style=none] (27) at (-2.75, -0.5) {};
		\node [style=none] (28) at (-1.75, 1) {};
		\node [style=none] (29) at (-2.5, -0.5) {};
		\node [style=none] (30) at (-1.75, 0) {};
		\node [style=none] (31) at (-1.75, -1) {};
		\node [style=none] (32) at (-1.25, 0.5) {};
		\node [style=none] (33) at (-1.25, -1.5) {};
		\node [style=none] (34) at (-1.75, -0.5) {};
		\node [style=none] (35) at (-1.25, 0) {};
		\node [style=none] (36) at (-1.25, -1) {};
		\node [style=none] (37) at (-2, -0.5) {};
		\node [style=none] (38) at (-1, 0) {};
		\node [style=none] (39) at (-1, -1) {};
		\node [style=none] (40) at (-1, 1) {};
		\node [style=none] (41) at (-3.75, 0) {};
		\node [style=none] (42) at (0.75, 0.5) {};
		\node [style=none] (43) at (0.75, -0.5) {};
		\node [style=none] (44) at (0.75, 0) {};
		\node [style=none] (45) at (0.25, 0) {};
		\node [style=none] (46) at (-0.75, 0.5) {};
		\node [style=none] (47) at (-0.75, -0.5) {};
		\node [style=none] (48) at (-0.75, 0) {};
		\node [style=none] (49) at (-0.25, 0) {};
	\end{pgfonlayer}
	\begin{pgfonlayer}{edgelayer}
		\draw (0.center) to (2.center);
		\draw (2.center) to (3.center);
		\draw (3.center) to (1.center);
		\draw (1.center) to (0.center);
		\draw (9.center) to (11.center);
		\draw (11.center) to (12.center);
		\draw (12.center) to (10.center);
		\draw (10.center) to (9.center);
		\draw (13.center) to (16.center);
		\draw (16.center) to (8.center);
		\draw [in=-180, out=0] (7.center) to (5.center);
		\draw (8.center) to (6.center);
		\draw [style=bluefill] (18.center) to (15.center);
		\draw [style=bluefill] (19.center) to (7.center);
		\draw [style=bluefill] (4.center) to (20.center);
		\draw (21.center) to (23.center);
		\draw (23.center) to (24.center);
		\draw (24.center) to (22.center);
		\draw (22.center) to (21.center);
		\draw (30.center) to (32.center);
		\draw (32.center) to (33.center);
		\draw (33.center) to (31.center);
		\draw (31.center) to (30.center);
		\draw (34.center) to (37.center);
		\draw (37.center) to (29.center);
		\draw [in=0, out=180] (28.center) to (26.center);
		\draw (29.center) to (27.center);
		\draw (38.center) to (35.center);
		\draw [style=bluefill] (39.center) to (36.center);
		\draw [style=bluefill] (40.center) to (28.center);
		\draw [style=bluefill] (25.center) to (41.center);
		\draw [style=bluefill] (39.center) to (19.center);
		\draw [style=bluefill] (40.center) to (18.center);
		\draw (42.center) to (43.center);
		\draw (45.center) to (42.center);
		\draw (43.center) to (45.center);
		\draw (46.center) to (47.center);
		\draw (49.center) to (46.center);
		\draw (47.center) to (49.center);
		\draw (38.center) to (48.center);
		\draw (44.center) to (14.center);
	\end{pgfonlayer}
\end{tikzpicture}} $.
  Note that this is equal to the associator $\assoc_{W, \unitW, W}$.
\end{example}

We can now show that every morphism $f : A \to B$ in $\CatSW$ is equal to
$\canonical(A, B)$.

\begin{proposition}
  \label{proposition:all-morphisms-canonical}
  $f = \objunpack(A) \cp \objpack(B)$
  for all $f : A \to B$ in $\CatSW$.
\end{proposition}

\begin{proof}
  By induction.
  On the base case--generators--the proof is straightforward; we give it for
  identities and generators \generator{g-adapter} and \generator{g-big-adapter},
  with the proofs for inverse generators following by a symmetric argument.
  \begin{align*}
    \canonical(X, X)
      & = \objunpack(X) \cp \objpack(X)
      = \objpack(X)^{-1} \cp \objpack(X)
      = \id_{X} \\
    \canonical(\unitSW, \ObjD{\unitW})
      & = \objunpack(\unitSW) \cp \objpack(\ObjD{\unitW})
      = \stikzfig{empty} \cp \begin{tikzpicture}
	\begin{pgfonlayer}{nodelayer}
		\node [style=none] (29) at (0, 0.5) {};
		\node [style=none] (30) at (0, -0.5) {};
		\node [style=none] (31) at (0, 0) {};
		\node [style=none] (32) at (-0.5, 0) {};
		\node [style=none] (33) at (0.5, 0) {};
	\end{pgfonlayer}
	\begin{pgfonlayer}{edgelayer}
		\draw (29.center) to (30.center);
		\draw (32.center) to (29.center);
		\draw (30.center) to (32.center);
		\draw (31.center) to (33.center);
	\end{pgfonlayer}
\end{tikzpicture}
      = \stikzfig{g-s-adapter} \\
    \canonical(\ObjD{A} \TensorD \ObjD{B}, \ObjD{A \TensorC B})
      & = \objunpack(\ObjD{A} \TensorD \ObjD{B}) \cp \objpack(\ObjD{A \TensorC B})
      = \stikzfig{proof-preorder-base-case-Adapter-1}
      = \stikzfig{g-s-big-adapter}
  \end{align*}
  The composition of canonical morphisms is canonical:
  \begin{align*}
    \canonical(X, Y) \cp \canonical(Y, Z)
      & = \objunpack(X) \cp \objpack(Y) \cp \objunpack(Y) \cp \objpack(Z) \\
      & = \objunpack(X) \cp \objpack(Y) \cp \objpack(Y)^{-1} \cp \objpack(Z) \\
      & = \objunpack(X) \cp \objpack(Z) \\
      & = \canonical(X, Z)
  \end{align*}
  And so is the tensor product:
  \begin{align*}
    \canonical(X_1, Y_1) \TensorD \canonical(X_2, Y_2)
      & = \scalebox{0.75}{\begin{tikzpicture}
	\begin{pgfonlayer}{nodelayer}
		\node [style=none] (0) at (-4.5, 1.25) {};
		\node [style=none] (1) at (-4.5, 0.25) {};
		\node [style=none] (2) at (-0.5, 0.25) {};
		\node [style=none] (3) at (-0.5, 1.25) {};
		\node [style=none] (4) at (-2.5, 0.75) {$\objunpack(X_1)$};
		\node [style=none] (5) at (-4.5, 0.75) {};
		\node [style=none] (6) at (-0.5, 0.75) {};
		\node [style=none] (7) at (-0.25, 0.75) {};
		\node [style=none] (8) at (-5, 0.75) {};
		\node [style=none] (9) at (-4.5, -0.25) {};
		\node [style=none] (10) at (-4.5, -1.25) {};
		\node [style=none] (11) at (-0.5, -1.25) {};
		\node [style=none] (12) at (-0.5, -0.25) {};
		\node [style=none] (13) at (-2.5, -0.75) {$\objunpack(X_2)$};
		\node [style=none] (14) at (-4.5, -0.75) {};
		\node [style=none] (15) at (-0.5, -0.75) {};
		\node [style=none] (16) at (-0.25, -0.75) {};
		\node [style=none] (17) at (-5, -0.75) {};
		\node [style=none] (18) at (0.5, 1.25) {};
		\node [style=none] (19) at (0.5, 0.25) {};
		\node [style=none] (20) at (4, 0.25) {};
		\node [style=none] (21) at (4, 1.25) {};
		\node [style=none] (22) at (2.25, 0.75) {$\objpack(Y_1)$};
		\node [style=none] (23) at (0.5, 0.75) {};
		\node [style=none] (24) at (4, 0.75) {};
		\node [style=none] (25) at (4.5, 0.75) {};
		\node [style=none] (26) at (0.25, 0.75) {};
		\node [style=none] (27) at (0.5, -0.25) {};
		\node [style=none] (28) at (0.5, -1.25) {};
		\node [style=none] (29) at (4, -1.25) {};
		\node [style=none] (30) at (4, -0.25) {};
		\node [style=none] (31) at (2.25, -0.75) {$\objpack(Y_2)$};
		\node [style=none] (32) at (0.5, -0.75) {};
		\node [style=none] (33) at (4, -0.75) {};
		\node [style=none] (34) at (4.5, -0.75) {};
		\node [style=none] (35) at (0.25, -0.75) {};
	\end{pgfonlayer}
	\begin{pgfonlayer}{edgelayer}
		\draw (0.center) to (3.center);
		\draw (3.center) to (2.center);
		\draw (2.center) to (1.center);
		\draw (1.center) to (0.center);
		\draw (6.center) to (7.center);
		\draw (8.center) to (5.center);
		\draw (9.center) to (12.center);
		\draw (12.center) to (11.center);
		\draw (11.center) to (10.center);
		\draw (10.center) to (9.center);
		\draw (15.center) to (16.center);
		\draw (17.center) to (14.center);
		\draw (18.center) to (21.center);
		\draw (21.center) to (20.center);
		\draw (20.center) to (19.center);
		\draw (19.center) to (18.center);
		\draw (24.center) to (25.center);
		\draw (26.center) to (23.center);
		\draw (27.center) to (30.center);
		\draw (30.center) to (29.center);
		\draw (29.center) to (28.center);
		\draw (28.center) to (27.center);
		\draw (33.center) to (34.center);
		\draw (35.center) to (32.center);
		\draw (16.center) to (35.center);
		\draw (7.center) to (26.center);
	\end{pgfonlayer}
\end{tikzpicture}} \\
      & = \scalebox{0.75}{\begin{tikzpicture}
	\begin{pgfonlayer}{nodelayer}
		\node [style=none] (0) at (-5.75, 0.5) {};
		\node [style=none] (1) at (-5.75, -0.5) {};
		\node [style=none] (2) at (-0.25, -0.5) {};
		\node [style=none] (3) at (-0.25, 0.5) {};
		\node [style=none] (4) at (-3, 0) {$\objunpack(X_1 \TensorD X_2)$};
		\node [style=none] (5) at (-5.75, 0) {};
		\node [style=none] (6) at (-0.25, 0) {};
		\node [style=none] (7) at (0, 0) {};
		\node [style=none] (8) at (-6.25, 0) {};
		\node [style=none] (18) at (0.5, 0.5) {};
		\node [style=none] (19) at (0.5, -0.5) {};
		\node [style=none] (20) at (5, -0.5) {};
		\node [style=none] (21) at (5, 0.5) {};
		\node [style=none] (22) at (2.75, 0) {$\objpack(Y_1 \TensorD Y_2)$};
		\node [style=none] (23) at (0.5, 0) {};
		\node [style=none] (24) at (5, 0) {};
		\node [style=none] (25) at (5.5, 0) {};
		\node [style=none] (26) at (0.25, 0) {};
	\end{pgfonlayer}
	\begin{pgfonlayer}{edgelayer}
		\draw (0.center) to (3.center);
		\draw (3.center) to (2.center);
		\draw (2.center) to (1.center);
		\draw (1.center) to (0.center);
		\draw (6.center) to (7.center);
		\draw (8.center) to (5.center);
		\draw (18.center) to (21.center);
		\draw (21.center) to (20.center);
		\draw (20.center) to (19.center);
		\draw (19.center) to (18.center);
		\draw (24.center) to (25.center);
		\draw (26.center) to (23.center);
		\draw (7.center) to (26.center);
	\end{pgfonlayer}
\end{tikzpicture}} \\
      & = \canonical(X_1 \TensorD X_2, Y_1 \TensorD Y_2) \qedhere
  \end{align*}
\end{proof}

\begin{proposition} $\CatSW$ is a preorder.
  \label{proposition:catsw-preorder}
\end{proposition}

\begin{proof}
  By Proposition \ref{proposition:catsw-morphisms-preserve-size} we know that all morphisms
  $f : A \to B$ have the property that $\objsize(A) = \objsize(B)$.
  We then define for any such objects a canonical morphism $\canonical(A, B)$ in
  Definition \ref{definition:canonical-arrows}.
  This canonical isomorphism is unique by Definition \ref{proposition:all-morphisms-canonical},
  and so $\CatSW$ is a preorder.
\end{proof}

Since we have now proven that $\CatSW$ is a preorder, it is now straightforward
to prove Theorem \ref{theorem:coherence}.
Note that this is essentially the opposite of the approach taken by Mac Lane,
who \emph{defines} a preorder, and then shows the existence of a unique strict
monoidal functor.

\begin{proof} (Proof of Theorem \ref{theorem:coherence}) \\
  By Proposition \ref{proposition:catw-unique-strict-monoidal-functor} there
  is a unique, strict monoidal functor from $\CatW$ to an arbitrary monoidal
  category $\CatM$ with $W \mapsto A$ for some $A \in \CatM$.
  Moreover, $\CatSW$ is a preorder, and so by Proposition
  \ref{proposition:if-catd-preorder-then-catc-preorder}, so is $\CatW$.
\end{proof}

A first consequence of the coherence theorem is that
$\FunG$ is a \emph{strict} inverse to $\FunF$ for morphisms
$f : \ObjD{A} \to \ObjD{B}$ in $\CatSW$.

\begin{proposition} If $f : \ObjD{A} \to \ObjD{B}$ then $\FunF(\FunG(f)) = f$.
  \label{proposition:fung-partial-inverse}
\end{proposition}

\begin{proof}
  We know that for any $A \in \CatW$ that $\FunG(\ObjD{A}) = A$.
  Thus for $f : \ObjD{A} \to \ObjD{B}$ we have $\FunG(f) : A \to B$
  and thus $\FunF(\FunG(f)) : \ObjD{A} \to \ObjD{B}$.
  But $\CatSW$ is a preorder, so we have
  $\FunF(\FunG(f)) = f$.
\end{proof}

Proposition \ref{proposition:fung-partial-inverse} guarantees that any morphism
of this type formed from adapters genuinely represents a specific morphism in
$\CatD$ built from associators and unitors; we can use this fact in to restate
the coherence theorem in terms of adapter morphisms.
Mac Lane's corollary then follows straightforwardly: the basic idea is to
`export' commuting diagrams from $\CatW$ to an arbitrary monoidal category by
interpreting the objects of $\CatW$ as functors, and arrows as natural
transformations.
We provide a graphical exposition of this proof in
Section~\ref{section:corollary-proof}.

\section{Coherence Corollary}
\label{section:corollary-proof}

\topicsentence{
  We can now state and prove Mac Lane's corollary~\cite[p.~169]{CFWM} to the coherence
  theorem.
}
Note that whereas this proof of the corollary is just a reformulation of Mac
Lane's argument in diagrammatic terms, the previous proof of Theorem
\ref{theorem:coherence} differs significantly.

\topicsentence{
  Let us begin with an informal statement of the theorem.
}
Take a commuting diagram of $\CatW$, for instance the triangle axiom below left
\eqref{equation:example-catw-triangle}:
\begin{multicols}{2}
  \begin{equation}
    \label{equation:example-catw-triangle}
    \scalebox{0.8}{
      \begin{tikzpicture}[node distance=2cm]
        \node (A) [					 ] {$W \TensorC (\unitC \TensorC W)$};
        \node (B) [right of=A] {};
        \node (C) [right of=B] {$(W \TensorC \unitC) \TensorC W$};
        \node (D) [below of=B] {$W \TensorC W$};

        \draw [->,thick] (A) to [] node[above] {$\assoc_{W,\unitC,W}$} (C);
        \draw [->,thick] (A) to [] node[below left ] {$\id_W \TensorC \unitl_W$} (D);
        \draw [->,thick] (C) to [] node[below right] {$\unitr_W \TensorC \id_W$} (D);
      \end{tikzpicture}
    }
  \end{equation}

  \begin{equation}
    \label{equation:example-catm-triangle}
    \scalebox{0.8}{
      \begin{tikzpicture}[node distance=2cm]
        \node (A) [					 ] {$A \TensorC (\unitC \TensorC B)$};
        \node (B) [right of=A] {};
        \node (C) [right of=B] {$(A \TensorC \unitC) \TensorC B$};
        \node (D) [below of=B] {$A \TensorC B$};

        \draw [->,thick] (A) to [] node[above] {$\assoc_{A,\unitC,B}$} (C);
        \draw [->,thick] (A) to [] node[below left ] {$\id_A \TensorC \unitl_B$} (D);
        \draw [->,thick] (C) to [] node[below right] {$\unitr_A \TensorC \id_B$} (D);
      \end{tikzpicture}
    }
  \end{equation}
\end{multicols}
The coherence theorem allows one to `export' this diagram to an arbitrary
monoidal category $\CatM$ by replacing each $i^{th}$ occurrence of W in a vertex
with some $A_i$ in $\CatM$.
For instance, let $A$ and $B$ be $\CatM$ objects, then we substitute the first
occurrence of $W$ in each vertex for $A$, and the second for $B$, giving us the
following commuting diagram in $\CatM$ on the right
\eqref{equation:example-catm-triangle}.

\begin{remark}
  The coherence theorem does \textbf{not} say that diagrams in
  $\CatM$ whose edges are components of natural transformations all commute;
  only those which correspond to diagrams in $\CatW$.  Put another way, if we
  have parallel $\CatM$-arrows $f, g : A \to B$ such that $f, g$ are constructed
  from associators and unitors, we may not in general conclude that $f = g$.
\end{remark}

\topicsentence{
  Now, it is not immediately obvious how even this informal coherence result
  follows from the statement of Theorem~\ref{theorem:coherence}.
}
Although for some fixed object $X \in \CatM$ there is a unique, strict monoidal
functor $\FunU : \CatW \to \CatM$,
this does not let us obtain every diagram we would like.
In particular, using $\FunF$ in this way we cannot obtain diagrams with multiple
variables such as \eqref{equation:example-catm-triangle}--only those where
\emph{every} $W$ is replaced by $X$.

\topicsentence{
  To allow for diagrams with multiple variables, Mac Lane constructs the
  non-strict monoidal category $\ItM$.
}

\begin{definition}[{$\ItM$, \cite[p.~169]{CFWM}}]
  Fix an arbitrary monoidal category $(\CatM, \TensorM, \unitM, \assoc, \unitl, \unitr)$.
  Then $\ItM$ is the category with:
  \begin{enumerate}
    \item Objects: functors $\CatM^n \to \CatM$
    \item Arrows: natural transformations
  \end{enumerate}
  With $\CatM^n$ denoting the $n$-fold product $\CatM \times \overset{n}{\ldots} \times \CatM$
\end{definition}

\begin{proposition} $\ItM$ is a (non-strict) monoidal category (from~\cite[p.~169]{CFWM}) \\
  \label{proposition:itm-monoidal}
  The monoidal unit is the constant functor $\functor{Const}_{\unit} : \CatT \to \CatM$.
  The monoidal product $\TensorItM : \ItM \times \ItM \to \ItM$ is defined on
  objects (functors) as:
  \[ \Functor{F} \TensorItM \Functor{G} = \stikzfig{itm-tensor-objects} \]
  and pointwise on arrows
  $\eta : \Functor{F}_1 \to \Functor{G}_1$ and
  $\mu : \Functor{F}_2 \to \functor{G}_2$
  so that for $\Functor{F}_1, \Functor{G}_1 : \CatM^n \to \CatM$
  and $\Functor{F}_2, \Functor{G}_2 : \CatM^m \to \CatM$
  the component at $A \times B \in \CatM^n \times \CatM^m$ is
  \begin{align*}
    (\eta \TensorItM \mu)_{A \times B}
        = \left[ \scalebox{0.9}{\stikzfig{itm-tensor-arrows-lhs}}\right ]_{A \times B}
        = \scalebox{1}{\stikzfig{itm-tensor-arrows-rhs}}
  \end{align*}
  Associators and unitors are similarly defined pointwise, i.e.:
  \begin{gather*}
    \begin{aligned}
      \unitl_{\Functor{F}_A}
        & = \left[ \scalebox{1}{\stikzfig{itm-unitl-lhs}} \right]_A \\
        & = \scalebox{1}{\stikzfig{itm-unitl-rhs}} \\
    \end{aligned}
    \qquad
    \begin{aligned}
      \unitr_{\Functor{F}_A}
        & = \scalebox{1}{\stikzfig{itm-unitr-rhs}} \\
    \end{aligned}
    \\[2em]
    \begin{aligned}
      \assoc_{{\Functor{F}, \Functor{G}, \Functor{H}}_{A,B,C}}
        & = \scalebox{1}{\stikzfig{itm-assoc-rhs}} \\
    \end{aligned}
  \end{gather*}
\end{proposition}

\begin{proof}
  Associators and unitors are natural since each of their components is natural.
  That is, given a natural transformation
  $\mu : \Functor{F} \to \Functor{G}$
  we know that
  $\unitr_\Functor{F} \cp \mu = (\mu \TensorItM \id) \cp \unitr_\Functor{G}$
  precisely because components are both sides are always equal, i.e for all $A$ we have
  $\unitr_{\Functor{F}_A} \cp \mu_A = (\mu \TensorItM \id)_A \cp \unitr_{\Functor{G}_A}$.
  A similar argument applies to $\assoc$ and $\unitl$.
  Further, the axioms of monoidal categories are satisfied for the same reason:
  each diagram commutes because all its \emph{components} commute using the
  monoidal structure of $\CatM$.
\end{proof}


This will allow us to regard objects $A \in \CatW$ of size $n$ as \emph{functors}
by applying the monoidal functor $\FunU(A) : \CatM^n \to \CatM$ as follows:
\begin{equation}
  \label{equation:u-catw-to-itm}
  \unitM        \mapsto \stikzfig{generator-functor-adapter}
  \qquad
  W             \mapsto \stikzfig{generator-functor-id}
  \qquad
  A \TensorW B  \mapsto \stikzfig{generator-functor-recursive}
\end{equation}

In the above, $\stikzfig{g-functor-const}$ denotes the constant functor
$\functor{Const}_{\unit} : \CatT \to \CatM$ mapping the single object of $\CatT$
to the monoidal unit $\unitM$.

Now, $\FunU : \CatW \to \ItM$ preserves diagrams since it is a functor, and so
we may picture the triangle axiom in $\ItM$ graphically as below.
\begin{equation}\label{eq:im}
	  \raisebox{0.9cm}{\scalebox{0.8}{\input{example-triangle-axiom-graphically.tikz}}}
\end{equation}
Note the use of graphical notation is justified by Proposition~\ref{proposition:itm-monoidal}.
 In~\eqref{eq:im}, vertices (in blue) depict functors, and edges depict natural
transformations.
This transformation of $\CatW$-objects to functors formalises the intuition of
`replacing the $i^{th}$ occurrence of $W$ in a diagram'.
That is, for a given diagram in $\CatW$ with vertices $V_i$ of size $n$,
we now simply make a particular choice of $\CatM^n$-object for each vertex
and apply $\FunU(V_i) : \CatM^n \to \CatM$ to obtain a `multivariable' diagram
in $\CatM$.

\topicsentence{
  Mac Lane then states the coherence result corollary as follows:
}

\begin{corollaryC}[{\cite[p.~169]{CFWM}}]
  Let $\cat{M}$ be a monoidal category. There is a function which assigns to each pair of objects
  $A, B \in \CatW$ of $\objsize$ $n$ a (unique) natural isomorphism
  $ \canonical_{\CatM}(A, B) : \FunU(A) \to \FunU(B) $
  called the canonical map from $\FunU(A)$ to $\FunU(B)$,
  in such a way that the identity arrow $\unitItM \to \unitItM$ is canonical
  (between functors of $0$ variables)
  the identity transformation $\id : \id_{\cat{M}} \to \id_{\cat{M}}$ is
  canonical,
  $\assoc, \unitl, \unitr$ (and their inverses) are canonical,
  and the composite and $\TensorItM$-product of canonical maps is canonical.
\end{corollaryC}

\begin{proof}[Proof from {\cite[p.~169]{CFWM}}]
  Let $\FunU : \CatW \to \ItM$ be the unique strict monoidal functor mapping $W$
  to the identity functor $\id : \CatM \to \CatM$ so that $\FunU$ acts on
  objects as in \eqref{equation:u-catw-to-itm}.
  Then $\FunU$ acts on morphisms of $\CatW$ as follows:
  \[
    \id_{\unitW}    \mapsto \id               
    \qquad \qquad
    \id_{W}         \mapsto \id               
  \]
  \[
    \unitl_A        \mapsto \unitl_{\FunU(A)} 
    \qquad \qquad
    \unitr_A        \mapsto \unitr_{\FunU(A)} 
    \qquad \qquad
    \assoc_{A,B,C}  \mapsto \assoc_{\FunU(A), \FunU(B), \FunU(C)} 
  \]
  \[
    f \TensorW g    \mapsto \FunU(f) \TensorItM \FunU(g) 
  \]
  And so $\canonical_{\CatM}(A, B) = \FunU(f)$ for each unique $f : A \to B$.
\end{proof}

Finally, note that the canonical morphism $\canonical_{\CatM}(A, B)$ can be
defined as
$ \canonical_{\CatM}(A, B) = (\FunU \circ \FunG)(\canonical(A, B)) $
thus we may use the normal form $\canonical(A, B)$ to determine the canonical
natural isomorphism in $\ItM$.

\section{Conclusions}
\label{section:conclusions}
The body of work on string diagrams in general is broad and growing rapidly. 
It is therefore slightly surprising that the fundamental issue of non-strict tensorial composition has not received more attention.
On the one hand, this is reasonable.
The assumption of strictness does not entail a loss of generality, as indeed we have confirmed via an example in Sec.~\ref{section:exnonstrict}.
However, non-strict tupling is a basic feature of programming languages, and even hardware description languages, and modelling it using string diagrams requires the proper mathematical framework.

This framework, the main contribution of the paper given in Def.~\ref{definition:catd}, shows a way to strictify a possibly non-strict monoidal category.
The body of the paper proves that the definition has all the desired properties and, in the process, we discuss two new proofs for Mac Lane's strictness and coherence theorems, respectively.
We believe that, as is usually the case, the string-diagrammatic perspective has pedagogical value, lending concrete intuitions to what otherwise seems like an abstract symbolic exercise.

\subsection{Further work}

Lack of support for non-strict tensor limits the range of many applications of string diagrams.
The first immediate question to study the strictification recipe in the context of applications to programming languages with higher-order functions, such as high-level circuit synthesis~\cite{DBLP:conf/popl/Ghica07} or automatic differentiation~\cite{DBLP:journals/corr/abs-2107-13433}. These can be formulated using hierarchical string diagrams (based on `functorial boxes'~\cite{DBLP:conf/csl/Mellies06}) to represent the structure  of monoidal-closed and cartesian-closed categories, see~\cite{GhicaZanasi23} for an overview.
Similar considerations motivate the study of strictification of traced monoidal categories, which can be used as models of digital circuits~\cite{DBLP:conf/csl/GhicaJL17}.

Further, our construction expands the use of datastructures and algorithms
currently limited only to the strict case (e.g., \cite{wilson2021cost,WilsonZanasi23b}).
Such datastructures are typically based on graph or hypergraph representations
for performance reasons; applying our construction allows us to leverage those benefits essentially for free.
In cases where such datastructures and algorithms are proven correct, it may be
beneficial to reproduce the proofs in this paper in a formal theorem prover in
order to provide end-to-end verification of applications.

Finally, a formal understanding of non-strict monoidal categories may open the door to more graphical approaches for theorem proving.
Interactive graphical theorem provers using string diagrams for strict monoidal categories such as \url{homotopy.io} represent a refreshingly new approach to the design of proof assistants.
Since models of, for example, intuitionistic logic are non-strict, the novel string diagrams in this paper could be used perhaps to develop similar graphical proof assistant for more conventional logics.

\paragraph{Acknowledgements} The second and third author acknowledge support from \textsc{esprc} grant EP/V002376/1. The third author also acknowledge support from \textsc{miur} PRIN P2022HXNSC (European Union - Next Generation EU).

\bibliographystyle{alphaurl}
\bibliography{main}

\appendix

\section{Sequential Normal Form}
\label{section:sequential-normal-form}
The following proposition is well-known (see for example
\cite{lafont_circuits}) and straightforward to prove, but we provide a proof
anyway for completeness.

\begin{proposition}
  \label{proposition:sequential-normal-form}
  let $\cat{C}$ be a monoidal category presented by generators $\Sigma$ and some equations.
  Then any (finite) term $t$ representing a morphism of $\cat{C}$ can be factored into `slices':
  \[ (\id \otimes \epsilon_1 \otimes \id) \cp (\id \otimes \epsilon_2 \otimes \id) \cp \ldots \cp (\id \otimes \epsilon_n \otimes \id) \]
  where each $\epsilon_i$ is a generator.
\end{proposition}

This factorization can be diagrammed as follows:

\[ \scalebox{0.9}{\begin{tikzpicture}
	\begin{pgfonlayer}{nodelayer}
		\node [style=none] (0) at (-6, 1) {};
		\node [style=none] (1) at (-4, 1) {};
		\node [style=none] (2) at (-6, 0) {};
		\node [style=none] (3) at (-4, 0) {};
		\node [style=none] (4) at (-6, -1) {};
		\node [style=none] (5) at (-4, -1) {};
		\node [style=none] (6) at (0, 1) {};
		\node [style=none] (7) at (2, 1) {};
		\node [style=none] (8) at (0, 0) {};
		\node [style=none] (9) at (2, 0) {};
		\node [style=none] (10) at (0, -1) {};
		\node [style=none] (11) at (2, -1) {};
		\node [style=none] (12) at (-2, 0) {$\cp$};
		\node [style=morphism] (13) at (-5, 0) {$\epsilon_1$};
		\node [style=morphism] (14) at (1, 0) {$\epsilon_2$};
		\node [style=none] (15) at (5, 0) {$\ldots$};
		\node [style=none] (16) at (8, 1) {};
		\node [style=none] (17) at (10, 1) {};
		\node [style=none] (18) at (8, 0) {};
		\node [style=none] (19) at (10, 0) {};
		\node [style=none] (20) at (8, -1) {};
		\node [style=none] (21) at (10, -1) {};
		\node [style=morphism] (22) at (9, 0) {$\epsilon_n$};
		\node [style=none] (23) at (4, 0) {$\cp$};
		\node [style=none] (24) at (6, 0) {$\cp$};
		\node [style=none] (25) at (-7, 1) {$X_1$};
		\node [style=none] (26) at (-7, 0) {$A_1$};
		\node [style=none] (27) at (-7, -1) {$Y_1$};
		\node [style=none] (28) at (-3, 1) {$X_1$};
		\node [style=none] (29) at (-3, 0) {$B_1$};
		\node [style=none] (30) at (-3, -1) {$Y_1$};
		\node [style=none] (31) at (-1, 1) {$X_2$};
		\node [style=none] (32) at (-1, 0) {$A_2$};
		\node [style=none] (33) at (-1, -1) {$Y_2$};
		\node [style=none] (34) at (3, 1) {$X_2$};
		\node [style=none] (35) at (3, 0) {$B_2$};
		\node [style=none] (36) at (3, -1) {$Y_2$};
		\node [style=none] (37) at (7, 1) {$X_n$};
		\node [style=none] (38) at (7, 0) {$A_n$};
		\node [style=none] (39) at (7, -1) {$Y_n$};
		\node [style=none] (40) at (11, 1) {$X_n$};
		\node [style=none] (41) at (11, 0) {$B_n$};
		\node [style=none] (42) at (11, -1) {$Y_n$};
	\end{pgfonlayer}
	\begin{pgfonlayer}{edgelayer}
		\draw (0.center) to (1.center);
		\draw (2.center) to (13);
		\draw (13) to (3.center);
		\draw (5.center) to (4.center);
		\draw (10.center) to (11.center);
		\draw (8.center) to (14);
		\draw (14) to (9.center);
		\draw (6.center) to (7.center);
		\draw (16.center) to (17.center);
		\draw (18.center) to (22);
		\draw (22) to (19.center);
		\draw (20.center) to (21.center);
	\end{pgfonlayer}
\end{tikzpicture}} \]

Note that in general $X_i \neq X_{i+1}$ and so on- i.e., the generators need not be ``aligned'' in this factorization.
For example, we can have morphisms like the following:

\begin{example}
  \[ \scalebox{0.9}{\begin{tikzpicture}
	\begin{pgfonlayer}{nodelayer}
		\node [style=none] (0) at (-2, 2.25) {};
		\node [style=none] (1) at (0, 2.25) {};
		\node [style=none] (2) at (-2, 0) {};
		\node [style=none] (3) at (0, 0.75) {};
		\node [style=none] (4) at (-2, -2) {};
		\node [style=none] (5) at (0, -2) {};
		\node [style=morphism] (13) at (-1, 0) {$\epsilon_1$};
		\node [style=none] (43) at (0, -0.75) {};
		\node [style=none] (46) at (2.5, 1.5) {};
		\node [style=none] (47) at (0.5, 2.25) {};
		\node [style=none] (48) at (2.5, -0.75) {};
		\node [style=none] (49) at (0.5, -0.75) {};
		\node [style=morphism] (50) at (1.5, 1.5) {$\epsilon_2$};
		\node [style=none] (51) at (0.5, 0.75) {};
		\node [style=none] (52) at (2.5, -2) {};
		\node [style=none] (53) at (0.5, -2) {};
	\end{pgfonlayer}
	\begin{pgfonlayer}{edgelayer}
		\draw (0.center) to (1.center);
		\draw (2.center) to (13);
		\draw (5.center) to (4.center);
		\draw [in=180, out=15, looseness=1.25] (13) to (3.center);
		\draw [in=180, out=-15] (13) to (43.center);
		\draw (46.center) to (50);
		\draw (49.center) to (48.center);
		\draw [in=0, out=165] (50) to (47.center);
		\draw [in=0, out=-165] (50) to (51.center);
		\draw (53.center) to (52.center);
		\draw (1.center) to (47.center);
		\draw (3.center) to (51.center);
		\draw (43.center) to (49.center);
		\draw (5.center) to (53.center);
	\end{pgfonlayer}
\end{tikzpicture}} \]
\end{example}

\begin{proof}
  We proceed by induction on terms.
  Let $S_0$ denote the set of terms consisting of identities and generators,
  Then let
  \[ S_n = S_0 \cup \{ t \cp u | t, u \in S_{n-1} \} \cup \{ t \otimes u | t, u \in S_{n-1} \} \]

  It is clear that terms in $S_0$ are already in sequential normal form, so it remains to prove the inductive case,
  beginning with composition.
  Let $v$ be a term in $S_{n+1}$.
  Now by inductive hypothesis, any term in $w \in S_n$ has an equivalent term in sequential normal form, which we'll denote
  $\hat{w}$.
  Now there are three cases:
  \begin{enumerate}
    \item If $v \in S_n$, then we have $\hat{v}$ by inductive hypothesis.
    \item If $v = t \cp u$, then $\hat{t}$ and $\hat{u}$ exist by inductive hypothesis, and we can form $\hat{v} = \hat{t} \cp \hat{u}$.
    \item If $v = t \otimes u$, then $\hat{v} = (\hat{t} \otimes \id) \cp (\id \otimes \hat{u})$
  \end{enumerate}
  and the proof is complete.
\end{proof}

\section{\texorpdfstring{$\FunG$}{N} is well-defined}
\label{section:g-well-defined}
In this appendix we verify that $\FunG$ is well-defined.
This amounts to two things: first that $\FunG$ is well-defined with respect to the interchange law,
and second that it respects the equations of Definition \ref{definition:catd}.

In the former case, sequential normal forms are only unique up to interchange,
so it must be verified that $\FunG$ preserves this property.
This is a straightforward if tedious exercise, which can be done by verifying
each of the cases in Definition \ref{definition:fung}.
Essentially, the only axioms required are naturality and equations
\cite[2.12, 2.13]{tensor_categories}.

Finally, we need to verify the equations of \ref{definition:catd}.
Specifically, for each of the monoidal, adapter, and associator/unitor equations
$\mathtt{lhs} = \mathtt{rhs}$, we show that $\FunG(\mathtt{lhs}) =
\FunG(\mathtt{rhs})$. We give derivations for these below.
Using these derivations one can also tediously check that the equations hold for
cases
$\id \TensorD \mathtt{lhs} \TensorD \id = \id \TensorD \mathtt{rhs} \TensorD \id$,
and so $\FunG$ is equal under any rewrite involving those equations; the only
cases of interest are for associator and unitor equations, which require the use
of the pentagon and triangle axioms, respectively.

We begin with the functor equations \eqref{equation:catd-functor-equations}
\begin{equation*}
  \begin{aligned}
  \FunG(\MorD{\id_A}) = \id_A = \FunG(\id_{\ObjD{A}})
  \end{aligned}
  \qquad
  \qquad
  \begin{aligned}
  \FunG(\MorD{f} \cp \MorD{g}) = \FunG(\MorD{f}) \cp \FunG(\MorD{g}) = f \cp g = \FunG(\MorD{f \cp g})
  \end{aligned}
\end{equation*}
Now the adapter equations \eqref{equation:catd-adapter-equations}:
\begin{equation*}
\begin{aligned}
  \FunG(\bigadapter \cp (\MorD{f} \TensorD \MorD{g}) \cp \bigadapterInv)
    & = \FunG(\bigadapter) \cp \FunG(\MorD{f} \TensorD \MorD{g}) \cp \FunG(\bigadapterInv) \\
    & = \FunG(\MorD{f} \TensorD \MorD{g}) \\
    & = \FunG(\MorD{f} \TensorD \id) \cp \FunG(\id \TensorD \MorD{g}) \\
    & = (f \TensorC \id) \cp (\id \TensorC g) \\
    & = f \TensorC g \\
    & = \FunG(\MorD{f \TensorC g}) \\
\end{aligned}
\qquad
\begin{aligned}
  \FunG(\adapter \cp \adapterInv)
    & = \FunG(\adapter) \cp \FunG(\adapterInv) \\
    & = \id_{\unitC} \cp \id_{\unitC} \\
    & = \id_{\unitC} \\
    & = \FunG(\id_{\unitD}) \\
\end{aligned}
\end{equation*}
\begin{equation*}
\begin{aligned}
  \FunG(\bigadapterInv \cp \MorD{f \TensorC g} \cp \bigadapterInv)
    & = \FunG(\bigadapterInv) \cp \FunG(\MorD{f \TensorC g}) \cp \FunG(\bigadapter) \\
    & = \FunG(\MorD{f \TensorC g}) \\
    & = f \TensorC g \\
    & = (f \TensorC \id) \cp (\id \TensorC g) \\
    & = \FunG(\MorD{f} \TensorD \id) \cp \FunG(\id \TensorD \MorD{g}) \\
    & = \FunG((\MorD{f} \TensorD \id) \cp (\id \TensorD \MorD{g})) \\
    & = \FunG(\MorD{f} \TensorD \MorD{g}) \\
\end{aligned}
\qquad
\begin{aligned}
  \FunG(\adapterInv \cp \adapter)
    & = \FunG(\adapterInv) \cp \FunG(\adapter) \\
    & = \id_{\unitC} \cp \id_{\unitC} \\
    & = \id_{\unitC} \\
    & = \FunG(\MorD{\id_{\unitC}}) \\
    & = \FunG(\id_{\ObjD{\unitC}}) \\
\end{aligned}
\end{equation*}

Finally the associator/unitor equations \eqref{equation:catd-monoidal-equations}:

\begin{align*}
  \FunG(\bigadapterInv \cp (\id \TensorD \bigadapterInv) & \cp (\bigadapter \TensorD \id) \cp \bigadapter) \\
    & = \FunG(\bigadapterInv) \cp \FunG(\id \TensorD \bigadapterInv) \cp \FunG(\bigadapter \TensorD \id) \cp \FunG(\bigadapter) \\
    & = \id \cp \id \cp \assoc \cp \id \\
    & = \assoc \\
    & = \FunG(\MorD{\assoc}) \\
\end{align*}
\begin{align*}
  \FunG(\bigadapterInv \cp (\bigadapterInv \TensorD \id) & \cp (\id \TensorD \bigadapter) \cp \bigadapter) \\
    & = \FunG(\bigadapterInv) \cp \FunG(\bigadapterInv \TensorD \id) \cp \FunG(\id \TensorD \bigadapter) \cp \FunG(\bigadapter) \\
    & = \id \cp \assocInv \cp \id \cp \id \\
    & = \assocInv \\
    & = \FunG(\MorD{\assocInv}) \\
\end{align*}
\begin{align*}
  \FunG(\bigadapterInv \cp (\adapterInv \TensorD \id))
    & = \FunG(\bigadapterInv) \cp \FunG(\adapterInv \TensorD \id) \\
    & = \id \cp \unitl \\
    & = \unitl \\
    & = \FunG(\MorD{\unitl}) \\
\end{align*}
\begin{align*}
  \FunG((\adapter \TensorD \id) \cp \bigadapter)
    & = \FunG(\adapter \TensorD \id) \cp \FunG(\bigadapter) \\
    & = \unitlInv \cp \id \\
    & = \unitlInv \\
    & = \FunG(\MorD{\unitlInv}) \\
\end{align*}
\begin{align*}
  \FunG(\bigadapterInv \cp (\id \TensorD \adapterInv))
    & = \FunG(\bigadapterInv) \cp \FunG(\id \TensorD \adapterInv) \\
    & = \id \cp \unitr \\
    & = \unitr \\
    & = \FunG(\MorD{\unitr}) \\
\end{align*}
\begin{align*}
  \FunG((\id \cp \adapter) \cp \bigadapter)
    & = \FunG(\id \cp \adapter) \cp \FunG(\bigadapter) \\
    & = \unitrInv \cp \id \\
    & = \unitrInv \\
    & = \FunG(\MorD{\unitrInv}) \\
\end{align*}

Thus $\FunG$ is well-defined with respect to the monoidal equations.

\end{document}